\documentclass{article}

\usepackage{amssymb}
\usepackage{amsmath}
\usepackage{amsthm}
\usepackage{fullpage}
\theoremstyle{plain}
\newtheorem{theorem}{Theorem}
\newtheorem{proposition}{Proposition}
\newtheorem{lemma}{Lemma}
\newtheorem{corollary}{Corollary}
\newtheorem{conjecture}{Conjecture}
\theoremstyle{definition}
\newtheorem{definition}{Definition}
\theoremstyle{remark}
\newtheorem{remark}{Remark}
\newtheorem{example}{Example}

\begin{document}

\title{On affine maps on non-compact convex sets and some characterizations of finite-dimensional solid ellipsoids}
\author{Gen Kimura\footnote{College of Systems Engineering and Science, Shibaura Institute of Technology, 3-9-14 Shibaura, Minato-ku, Tokyo 108-8548, Japan (gen[at]sic.shibaura-it.ac.jp)}\quad and\quad Koji Nuida\footnote{Research Institute for Secure Systems (RISEC), National Institute of Advanced Industrial Science and Technology (AIST), AIST Tsukuba Central 2, 1-1-1 Umezono, Tsukuba, Ibaraki 305-8568, Japan (k.nuida[at]aist.go.jp)}\ \footnote{Corresponding author}}

\maketitle

\begin{abstract}
Convex geometry has recently attracted great attention as a framework to formulate general probabilistic theories.
In this framework, convex sets and affine maps represent the state spaces of physical systems and the possible dynamics, respectively.
In the first part of this paper, we present a result on separation of simplices and balls (up to affine equivalence) among all compact convex sets in two- and three-dimensional Euclidean spaces, which focuses on the set of extreme points and the action of affine transformations on it.
Regarding the above-mentioned axiomatization of quantum physics, our result corresponds to the case of simplest ($2$-level) quantum system.
We also discuss a possible extension to higher dimensions.
In the second part, towards generalizations of the framework of general probabilistic theories and several existing results including ones in the first part from the case of compact and finite-dimensional physical systems as in most of the literatures to more general cases, we study some fundamental properties of convex sets and affine maps that are relevant to the above subject.
\end{abstract}

\section{Introduction}
\label{sec:introduction}

\subsection{Backgrounds and our contributions}
\label{subsec:introduction_summary}

Convexity is a ubiquitous notion in mathematics, frequently appearing not only in geometry but also in other various research areas, including many applications to outside mathematics (see e.g., \cite{Mat_book,Stu_book} and references therein).
Among them, an interesting study of convexity has emerged in the foundations of quantum mechanics.
These activities aim at interpreting quantum physics as an instance of more general physical theories (called e.g., \lq\lq general probabilistic theories''), the latter being axiomatized from operational viewpoints, using \lq\lq (physical) states'' and \lq\lq measurements'' as the basic notions.
Here, probabilistic mixture of states are formalized as convex combination of states, therefore the notion of convexity is essential in those studies.
A motivation of studying such general theories is to establish a unified theoretical framework to describe quantum (and classical) physics together with its possible variants or generalizations.
Potential applications of such activities would include cryptography with long-standing security; if one wants to estimate the security of present cryptographic schemes against attacks using physical devices in the next 100 years, where the present quantum physics may be improved by some advanced theory, then such an observation of general physical theories may be of some help.
Another motivation is to characterize quantum physics among such general theories, giving a re-axiomatization of quantum physics based on physical principles, which is expected to be more physically intuitive than von Neumann's original axiom based (mysteriously) on Hilbert spaces.
For more introduction to this subject from physical viewpoints and several preceding works, see e.g., \cite{Ara_book,Gud_book,Hol_book,KN,Mac_book} and references therein.

A common philosophy in the above-mentioned studies of general physical theories can be understood as follows: We put mathematical assumptions that are essential (or inevitable) from physical viewpoints, while the quantity of other \lq\lq technical'' assumptions should be as small as possible.
To formulate the state space (the set of physical states) of such a general physical theory, the state space is conventionally assumed to be a convex set, reflecting the above-mentioned requirement that any probabilistic mixture of two states should also be a state.

In the first part of this paper, we investigate finite-dimensional compact state spaces (convex sets) equipped with symmetry and some additional special properties.
In the theory of convex polytopes, the notion of symmetry (precisely, vertex-transitivity of affine isometric transformation groups) has played one of the most significant roles and such symmetric convex polytopes have been intensively studied (e.g., \cite{Bab}).
Here, as usual, we say that a group $G$ {\em acts transitively} on a set $X$, if for any $x_1,x_2 \in X$, we have $g \cdot x_1 = x_2$ for some element $g \in G$.
In the studies of general physical theories, the symmetry property \cite{Dav} can also be considered as one of physical principles which can be interpreted as a possibility of reversible transformation between pure states \cite{DB,Hardy,Masanes} or as a physical equivalence of pure states \cite{KNI,KN}. 
In this context, the characterization of a state space under the hypothesis of symmetry becomes an important subject.
In this paper, we give the following characterizations on $2$- and $3$-dimensional compact convex sets with the symmetry property:
\begin{theorem}
\label{thm:2-dimensional_symmetric}
Let $\mathcal{S}$ be a compact convex subset of a Euclidean space with $2$-dimensional affine hull; $\mathrm{Aff}(\mathcal{S}) = \mathbb{R}^2$.
Then the group of bijective affine transformations of $\mathcal{S}$ acts transitively on the set $\mathcal{S}_{\mathrm{ext}}$ of extreme points of $\mathcal{S}$ (see above for the terminology) if and only if $\mathcal{S}$ is affine isomorphic to one of the following two kinds of objects:
\begin{enumerate}
\item A symmetric (or vertex-transitive) convex polygon.
\item The ($2$-dimensional) unit disk.
\end{enumerate}
\end{theorem}
\begin{theorem}
\label{thm:3-dimensional_symmetric}
Let $\mathcal{S}$ be a compact convex subset of a Euclidean space with $3$-dimensional affine hull; $\mathrm{Aff}(\mathcal{S}) = \mathbb{R}^3$.
Then the group of bijective affine transformations of $\mathcal{S}$ acts transitively on the set $\mathcal{S}_{\mathrm{ext}}$ of extreme points of $\mathcal{S}$ if and only if $\mathcal{S}$ is affine isomorphic to one of the following three kinds of objects:
\begin{enumerate}
\item A $3$-dimensional symmetric (or vertex-transitive) convex polytope.
\item A $3$-dimensional circular cylinder $\{{}^t(x,y,z) \in \mathbb{R}^3 \mid x^2 + y^2 \leq 1\,,\, 0 \leq z \leq 1\}$.
\item The $3$-dimensional unit ball.
\end{enumerate}
\end{theorem}

Here we give a remark on a related work: 
In a preceding paper \cite{Dav}, Davies also studied the finite-dimensional compact convex sets $\mathcal{S}$ whose groups $G$ of bijective affine transformations act transitively on the sets of extreme points, and presented a classification of the convex sets $\mathcal{S}$ such that its group $G$ of symmetry as above is equal to a given compact group, in terms of classification of finite-dimensional subspaces of the left regular representation of the given compact group.
However, in that work any relation between the symmetric convex sets with {\em different} groups of symmetry, and also the concrete shape of a symmetric convex set induced by each subspace of the regular representation, have not been clarified. 
Therefore, to classify {\em all} symmetric convex sets based on that result, it is basically required to determine the concrete shape of {\em every} convex set constructed from some representation space of {\em every} compact group.
On the other hand, our results above aimed at determining the symmetric convex sets without any restriction for the groups of symmetry. 

Next, we investigate another kind of physically motivated hypothesis, called the {\em spectrality} of state spaces.
First, we remind the definition of distinguishability of states in a single shot measurement (see, e.g., \cite{KNI} for its motivation from physical viewpoints):
\begin{definition}
\label{defn:distinguishability}
We say that points $s_1,s_2,\dots,s_n$ of a convex subset $\mathcal{S}$ of a real vector space are {\em distinguishable} if there exists a collection $(e_i)_{i=1}^{n}$ of $n$ affine functionals $e_i:\mathcal{S} \to \mathbb{R}$ such that $e_i \geq 0$, $\sum_{i=1}^{n} e_i = 1$ and $e_i(s_i) = 1$ for every $1 \leq i \leq n$.
\end{definition}
Note that any non-empty subset of a set of distinguishable points is also distinguishable.
Roughly speaking, if a convex set $\mathcal{S}$ has $n$ distinguishable points $s_1,\dots,s_n$, then \lq\lq lossless'' encoding of any $n$-bit information into a point of $\mathcal{S}$ is (in principle) possible by using these $n$ points.
(A geometric interpretation of this definition will be supplied in Lemma \ref{lem:distinguishable_states} below.)
Then we introduce the following definition:
\begin{definition}
[{see e.g., \cite{AS,IS}}]
\label{defn:distinguishably_decomposable}
Let $\mathcal{S}$ be a convex subset of a real vector space.
We say that $\mathcal{S}$ has {\em spectrality} (or, as in \cite{KN}, $\mathcal{S}$ is {\em distinguishably decomposable}) if each $s \in \mathcal{S}$ admits a decomposition $s = \sum_{j=1}^{\ell} \lambda_j s_j$ ($1 \leq \ell < \infty$) into distinguishable extreme points $s_1,\dots,s_{\ell} \in \mathcal{S}_{\mathrm{ext}}$ such that $\sum_{j=1}^{\ell} \lambda_j = 1$ and $\lambda_j \geq 0$ for every $j$.
Moreover, if the number $\ell$ of distinguishable extreme points in each decomposition is bounded above by $k$, then we say that $\mathcal{S}$ has {\em $k$-spectrality}.
\end{definition}
In general physical theories, the spectrality of state spaces can also be one of the physical principles which can be interpreted as the possibility of state preparation with a probabilistic mixtures of distinguishable pure states \cite{KN}.
Here we emphasize that the distinguishable points appearing in Definition \ref{defn:distinguishably_decomposable} should be {\em extreme points}, which rules out, e.g., the square and any regular polygon since a generic point in these convex sets cannot be decomposed into extreme points (even though such a point can be decomposed into boundary points; see Corollary \ref{cor:distinguishably_decomposable_simplex}).
By using this notion, in this paper we give the following enhancement of the above theorems, the latter of which separates (up to affine equivalence) the $3$-simplex and the $3$-dimensional ball among arbitrary $3$-dimensional convex sets:
\begin{theorem}
\label{thm:2-dimensional_strongly_symmetric}
Let $\mathcal{S}$ be a compact convex subset of a Euclidean space with $\mathrm{Aff}(\mathcal{S}) = \mathbb{R}^2$.
Then the following two conditions are equivalent:
\begin{enumerate}
\item The set $\mathcal{S}$ is affine isomorphic to either a triangle (i.e., $2$-simplex) or the unit disk.
\item The group of bijective affine transformations of $\mathcal{S}$ acts transitively on $\mathcal{S}_{\mathrm{ext}}$, and $\mathcal{S}$ has spectrality.
\end{enumerate}
\end{theorem}
\begin{theorem}
\label{thm:3-dimensional_strongly_symmetric}
Let $\mathcal{S}$ be a compact convex subset of a Euclidean space with $\mathrm{Aff}(\mathcal{S}) = \mathbb{R}^3$.
Then the following two conditions are equivalent:
\begin{enumerate}
\item The set $\mathcal{S}$ is affine isomorphic to either a tetrahedron (i.e., $3$-simplex) or the unit ball.
\item The group of bijective affine transformations of $\mathcal{S}$ acts transitively on $\mathcal{S}_{\mathrm{ext}}$, and $\mathcal{S}$ has spectrality.
\end{enumerate}
\end{theorem}
Note that a classical probability theory is characterized by a simplex state space, while in $2$-level quantum systems the state space is affinely isomorphic to a unit ball (called the Bloch ball; see e.g., \cite{Kim_book,NC_book}).
Thus, the result is physically important since we have shown that in $3$-dimensional state space the physical theories are restricted to be either classical or quantum under the two physical principles of symmetricity and spectrality \cite{KN}.
(See \cite{CAP_preprint,DB,Hardy,Masanes} for another characterizations of the Bloch ball.)
Here we emphasize that the dimension of a state space $\mathcal{S}$ is also (in principle) operationally determined.
Indeed, the dimension of $\mathcal{S}$ is the minimum number of two-outcome measurements that are sufficient to identify a state in $\mathcal{S}$ uniquely from the outcome probabilities (cf., \cite{CAP_preprint}).
Hence, the restriction of dimensions of the state spaces in our results can also be regarded as a possible \lq\lq physical principle''.

For higher (finite) dimensional cases, we also give the following result:
\begin{theorem}
\label{thm:characterization_finite_dim}
Let $\mathcal{S}$ be a finite-dimensional compact convex set with $\mathrm{Aff}(\mathcal{S}) = \mathbb{R}^n$, $n < \infty$.
Let $G$ denote the group of bijective affine transformations of $\mathcal{S}$.
Then $\mathcal{S}$ is affine isomorphic to the $n$-dimensional unit ball if and only if the following two conditions are satisfied:
\begin{enumerate}
\item The diagonal action of $G$ on the set of pairs $(s_1,s_2)$ of distinguishable extreme points $s_1,s_2 \in \mathcal{S}_{\mathrm{ext}}$ is transitive.
\item The set $\mathcal{S}$ has $2$-spectrality.
\end{enumerate}
\end{theorem}
This result provides a new characterization of finite-dimensional solid ellipsoids among all convex sets in terms of structure of the set of extreme points. 

Moreover, in fact the authors have the following conjecture:
\begin{conjecture}
\label{conj:characterization_finite_dim}
We would be able to weaken the first condition in Theorem \ref{thm:characterization_finite_dim} for the diagonal action of $G$ to the transitivity of $G$ on the extreme points (cf., Lemma \ref{lem:2-transitive_implies_transitive}); the transitivity of $G$ on $\mathcal{S}_{\mathrm{ext}}$ and the $2$-spectrality would characterize the affine isomorphism classes of finite-dimensional unit balls.
\end{conjecture}
By Theorems \ref{thm:2-dimensional_strongly_symmetric} and \ref{thm:3-dimensional_strongly_symmetric}, this conjecture is true for up to $3$-dimensional cases.
A study for a general finite-dimensional case will be a future research topic.

Before explaining our contributions in the second part of the paper, we note that most of the preceding works on the operational treatments of general physical theories adopted the following two assumptions; namely that the state space is finite dimensional and compact \cite{CAP_preprint,DB,Hardy,Masanes}.
Operationally, the assumption of compactness is typically justified \cite{Ara_book} by the fact that physical measurements have a finite accuracy, and therefore, it is natural to assume that the limit point of a sequence of physical states is also a physical state.
However, this does not mean that our world indeed has a compact state space, and thus it is still desirable to describe the general theories without this unnecessary hypothesis (at least from the mathematical point of view). 
The second assumption of the finite dimensionality of state space crucially restrict the theories for the description of physics --- for instance, the state space of an electron is infinite dimensional.
The aim of the second part of this paper is thus to investigate the cases where the compactness or finiteness of dimension is not satisfied (while the part also includes some results on compact and/or finite-dimensional state spaces as special cases). 

Let us give a further explanation of the content of the second part.
In the formulation of physical state spaces using convex sets, the notion of \lq\lq dynamics'' on physical systems can be formulated as affine maps between convex sets in order to preserve the probabilistic mixtures.
To study affine maps between two state spaces, we introduce compact closures of the state spaces and consider affine maps between the compact closures, as each of the former affine maps extends uniquely to some of the latter.
Hence, if we deal with the set of affine maps as a plain set, there is no problem to assume the state spaces to be compact.
However, when we define a topology on the set of affine maps like the compact-open topology according to the above-mentioned common philosophy, we do not want to use information on the boundaries of state spaces, which are introduced by just technical reasons.
From this viewpoint, we introduce a notion of \lq\lq essential'' open or closed subsets of the compact closure of state space, which means (roughly) that the essential shape of the subset is not affected by the existence of the boundary of the state space (see Definition \ref{defn:essential_subset} for the precise definition).
By using the notion of essential subsets, we introduce a topology, which is an analogy of the compact-open topology, on the set of continuous maps between the compact closures, into which the set of affine maps between the state spaces is naturally embedded (see Definition \ref{defn:modified_compact_open_topology}).
This new topology also has desirable properties; for example, the induced topology on the set of affine maps between state spaces is Hausdorff (Proposition \ref{prop:dynamics_Hausdorff}), relatively compact in finite-dimensional cases (Proposition \ref{prop:Atilde_dense_in_A_finite_dim} and Proposition \ref{prop:A_is_compact}), and compatible with natural algebraic operations and natural actions. 
The authors hope that this topology and the notion of essential subsets are not only physically reasonable but also of purely mathematical interest.

\subsection{Related work}
\label{subsec:introduction_related}

Here we compare the present work to some related results.
Hardy \cite{Hardy} characterized the quantum system by five axioms on probabilistic behaviors, minimality of dimension of the state space (called \lq\lq the number of degrees of freedom'' in that paper), properties of subsystems and composite systems, and symmetry of states.
In one direction, the result has more generality than our result, since it deals with general quantum systems.
However, those axioms cannot be considered as physical principles as each of them is not directly testable in experiments.
Moreover, one of the starting points in \cite{Hardy} is that the maximum cardinality, denoted here by $c$, of a set of distinguishable states for such a single system (called \lq\lq the dimension'' in that paper) is assumed to be two (see, in particular, Axiom 2 and Axiom 3 in \cite{Hardy}).
In contrast, our result (Theorem \ref{thm:3-dimensional_strongly_symmetric}) does not introduce any restriction on $c$ at the beginning; a general result (Lemma \ref{lem:maximal_distinguishable}) implies that $c \leq 4$ for a single system in our situation and the case $c = 4$ corresponds to the classical system, while the case $c = 3$ is not excluded in the statement and is successfully ruled out in the proof.
From the viewpoint, our result indeed gives a new insight for characterizing the quantum system.

Similar things are also true for other related results: In a result of Masanes and Mueller \cite{Masanes}, they characterized the quantum system by five requirements which would have similar flavor but are different from those in \cite{Hardy}.
Although their result covers general systems as well as the single one, they also started the argument from an assumption that the quantity $c$ for a single system (called \lq\lq the capacity'' in their paper) is two.
Daki\'{c} and Brukner \cite{DB} also put an axiom on \lq\lq the information carrying capacity'', which means in our terminology here that $c \leq 2$ holds for a single system.
Moreover, the axiom on the existence of an \lq\lq ideal compression scheme'' of each state proposed by Chiribella, D'Ariano and Perinotti in \cite{CAP_preprint} (in particular, the \lq\lq maximally efficient'' condition for the compression scheme), which is different from the previous work \cite{Hardy,Masanes}, also implies (together with the other axioms) a constraint on the possibility of the quantity $c$.
Although these works are more based on physical principles than the one in \cite{Hardy}, we emphasize again that our present result does not put any restriction on $c$ at the beginning of the argument.
We also note that, in the results above, some conditions for the joint systems are needed even to derive a property of the single systems; while our argument does not use joint systems and is closed within the single systems.

\subsection{Notations and terminology}
\label{subsec:notations}

Unless otherwise specified, in this paper any vector space is considered over the real field $\mathbb{R}$, and symbols $\mathcal{S}$, $\mathcal{S}_0$, $\mathcal{S}_1$, $\mathcal{S}_2,\dots$ denote non-empty convex subsets of some vector spaces.
The notation $\mathrm{Aff}(\mathcal{S})$ signifies the affine hull of $\mathcal{S}$.
For any pair of topological spaces $X$ and $Y$, let $C(X,Y)$ denote the set of all continuous maps from $X$ to $Y$.
For any subset $A \subset X$, let $\mathrm{int}_X(A)$ and $\mathrm{cl}_X(A)$ denote the interior and the closure of $A$ relative to $X$, respectively.
When the underlying space $X$ is obvious from the context, we instead write $A^o = \mathrm{int}_X(A)$, $\overline{A} = \mathrm{cl}_X(A)$ and $A^c = X \setminus A$.

\subsection{Organization of the paper}
\label{subsec:organization}

This paper is organized as follows.
In Section \ref{sec:characterization_ellipsoids} we prove the theorems listed in Section \ref{subsec:introduction_summary}; some preliminary observations are provided in Section \ref{subsec:characterization_preliminary}, and the bodies of the proofs are given in Sections \ref{subsec:characterization_finite_dim}--\ref{subsec:characterization_low_dimensional_symmetric}.
On the other hand, towards generalizations of several existing results on general probabilistic theories from the case of compact and finite-dimensional physical systems as in most of the literatures to more general cases, Section \ref{sec:preliminary} summarizes some basic definitions and fundamental properties relevant to convex sets, and introduce some notions for the study of non-compact convex sets.
Then in Section \ref{sec:dynamics_topology}, we study algebraic and topological properties of the sets of affine maps between convex sets; Section \ref{subsec:dynamics_topology_general} deals with general (possibly non-compact and infinite-dimensional) cases and Section \ref{subsec:finite_dimensional} is specialized to finite-dimensional cases.

\section{On characterizations of solid ellipsoids}
\label{sec:characterization_ellipsoids}

In this section, we discuss some characterizing properties of the unit balls up to affine equivalence mentioned in Section \ref{subsec:introduction_summary}.
Throughout this section, we suppose (unless otherwise specified) that $\mathcal{S}$ is a convex, compact and full-dimensional subset of a finite-dimensional Euclidean space denoted by $V(\mathcal{S})$.
Let $\mathcal{S}_{\mathrm{ext}}$ denote the set of extreme points of $\mathcal{S}$, and let $\mathrm{Aut}(\mathcal{S})$ denote the set of bijective affine transformations on $\mathcal{S}$.

\subsection{Preliminary observations}
\label{subsec:characterization_preliminary}

First we note that, in the current setting for $\mathcal{S}$, $\mathrm{Aut}(\mathcal{S})$ is a compact topological group with respect to the function composition and the standard compact-open topology, and its natural action on $\mathcal{S}$ is continuous (these facts can also be derived from the generalized argument in Section \ref{sec:dynamics_topology}).
We also note that the action of $\mathrm{Aut}(\mathcal{S})$ preserves the subset $\mathcal{S}_{\mathrm{ext}}$ of $\mathcal{S}$.
Now we have the following fundamental property:
\begin{proposition}
\label{prop:pure_states_symmetric_GPT}
Suppose that $\mathrm{Aut}(\mathcal{S})$ acts transitively on $\mathcal{S}_{\mathrm{ext}}$.
Then $\mathcal{S}_{\mathrm{ext}}$ is a compact (hence closed) subset of $\mathcal{S}$.
\end{proposition}
\begin{proof}
By the assumption, the $\mathrm{Aut}(\mathcal{S})$-orbit of any given $s_0 \in \mathcal{S}_{\mathrm{ext}}$ coincides with the whole $\mathcal{S}_{\mathrm{ext}}$.
Therefore, the compactness of $\mathcal{S}_{\mathrm{ext}}$ follows from the compactness of $\mathrm{Aut}(\mathcal{S})$ and the continuity of the $\mathrm{Aut}(\mathcal{S})$-action.
\end{proof}
Note that $\mathcal{S}_{\mathrm{ext}}$ is not always a closed subset of $\mathcal{S}$; see e.g., \cite[Section 2.4, Exercise 2]{Gru}.
\begin{remark}
\label{rem:extreme_is_exposed}
We say that a point $s \in \mathcal{S}$ is {\em exposed} if the set $\{s\}$ forms a face of $\mathcal{S}$.
In the current setting, this is equivalent to that there is an affine functional $f:\mathcal{S} \to \mathbb{R}$ such that $f(\mathcal{S})$ is bounded and $s$ is the unique point in $\mathcal{S}$ to attain the maximum value of $f$.
Any exposed point in $\mathcal{S}$ is also an extreme point in $\mathcal{S}$, and the action of $\mathrm{Aut}(\mathcal{S})$ also preserves the set of the exposed points.
On the other hand, when the condition for compactness of $\mathcal{S}$ is relaxed, it does {\em not} hold in general that every extreme point is exposed (see e.g., \cite[Section 2.4, Exercise 2]{Gru}).
However, every extreme point is exposed if $\mathcal{S}_{\mathrm{ext}} \neq \emptyset$ (which is satisfied when $\mathcal{S}$ is compact) and $\mathrm{Aut}(\mathcal{S})$ acts transitively on $\mathcal{S}_{\mathrm{ext}}$.
Indeed, it is known that the set of exposed points in $\mathcal{S}$ is dense in $\mathcal{S}_{\mathrm{ext}}$ (see e.g., \cite[Theorem 2.4.9]{Gru}); in particular, an exposed point exists in this case.
Let $s_0$ be an exposed point in $\mathcal{S}$, therefore $s_0 \in \mathcal{S}_{\mathrm{ext}}$.
Then the $\mathrm{Aut}(\mathcal{S})$-orbit of $s_0$ consists of exposed points, while this orbit coincides with $\mathcal{S}_{\mathrm{ext}}$ by the transitivity assumption, therefore every $s \in \mathcal{S}_{\mathrm{ext}}$ is exposed as desired.
\end{remark}

Since $\mathrm{Aut}(\mathcal{S})$ is a compact group in the current case, $\mathrm{Aut}(\mathcal{S})$ admits the (normalized) Haar integral $\int_{G} f(g)\,dg$ (where we put $G = \mathrm{Aut}(\mathcal{S})$) of continuous real functions $f$ which is invariant under both left- and right-translations (see e.g., \cite{HM_compact_group}).
By using the Haar integral, it has been shown in \cite[Lemma 1]{Dav} that such a convex set $\mathcal{S}$ has a unique fixed point of the action of $\mathrm{Aut}(\mathcal{S})$.
(For readers' convenience, in \ref{sec:appendix_maximally_mixed_state_symmetric_GPT} we summarize a proof of this fact based on the existence of left-invariant integral for real-valued continuous functions on compact groups.)
From now on, by choosing an appropriate coordinate system of $V(\mathcal{S})$, we assume without loss of generality that the fixed point is the origin $0$ of the vector space $V(\mathcal{S})$.
In this setting, the group $\mathrm{Aut}(\mathcal{S})$ acts on $\mathcal{S}$ as bijective {\em linear} transformations, which extend uniquely to bijective linear transformations on $V(\mathcal{S})$.
Moreover, the next lemma implies that the origin of $V(\mathcal{S})$ is now an interior point of $\mathcal{S}$:
\begin{lemma}
\label{lem:maximally_mixed_is_interior}
Suppose that $\mathrm{Aut}(\mathcal{S})$ acts transitively on $\mathcal{S}_{\mathrm{ext}}$.
Then a fixed point $s \in \mathcal{S}$ of the $\mathrm{Aut}(\mathcal{S})$-action satisfies that $s \in \mathrm{int}_{V(\mathcal{S})}(\mathcal{S})$.
\end{lemma}
\begin{proof}
Assume for contrary that $s \in \partial \mathcal{S}$.
Let $H \subset V(\mathcal{S})$ be a supporting hyperplane of $\mathcal{S}$ at $s$.
As $H$ is convex and $H \subsetneq V(\mathcal{S}) = \mathrm{Aff}(\mathcal{S})$, we have $\mathcal{S}_{\mathrm{ext}} \not\subset H$.
Choose a point $t \in \mathcal{S}_{\mathrm{ext}} \setminus H$.
On the other hand, choose a convex decomposition $s = \sum_{i=1}^{\ell} \lambda_i s_i$ of $s$ into extreme points $s_i \in \mathcal{S}_{\mathrm{ext}}$ such that $\lambda_i > 0$ for every $i$.
Then by the assumption on the $\mathrm{Aut}(\mathcal{S})$-action, we have $f(s_1) = t$ for some $f \in \mathrm{Aut}(\mathcal{S})$.
Now we have $s = f(s) = \sum_{i} \lambda_i f(s_i)$, $f(s_i) \in \mathcal{S}$ and $\lambda_i > 0$ for every $i$.
As $H$ is a supporting hyperplane of $\mathcal{S}$ at $s$, it follows that $f(s_i) \in H$ for every $i$.
However, this contradicts the above fact $f(s_1) = t \not\in H$.
Hence the claim holds.
\end{proof}
As mentioned in Section \ref{subsec:introduction_summary}, the symmetry (or vertex-transitivity) property for convex polytopes $\mathcal{S}$ is usually defined as the transitivity of the affine {\em isometry} group of $\mathcal{S}$ on the vertices (extreme points) of $\mathcal{S}$, while in our argument the members of $\mathrm{Aut}(\mathcal{S})$ are not supposed to be isometric.
However, the next result shows that the two kinds of symmetry are essentially the same (up to affine equivalence):
\begin{proposition}
\label{prop:isometry}
Suppose that $\mathrm{Aut}(\mathcal{S})$ acts transitively on $\mathcal{S}_{\mathrm{ext}}$ and the unique fixed point of the $\mathrm{Aut}(\mathcal{S})$-action is the origin of $V(\mathcal{S})$.
Then there exists a bijective linear transformation $\varphi$ on $V(\mathcal{S})$ such that $\mathcal{S}' = \varphi(\mathcal{S})$ is a compact convex subset of $V(\mathcal{S})$ with $\mathrm{Aff}(\mathcal{S}') = V(\mathcal{S})$, $\mathrm{Aut}(\mathcal{S}')$ acts transitively on $\mathcal{S}'_{\mathrm{ext}}$, and each member of $\mathrm{Aut}(\mathcal{S}')$ is an orthogonal linear transformation (hence an isometry) on $V(\mathcal{S})$ with respect to the standard inner product $\langle \cdot,\cdot \rangle$.
\end{proposition}

Although Proposition \ref{prop:isometry} would be a consequence of a kind of standard argument (see e.g., \cite{CAP_preprint,DLL,Masanes}), we give a proof of Proposition \ref{prop:isometry} in \ref{sec:appendix_proof_isometry} for the sake of completeness.
An easy but important consequence of the above observation is the following:
\begin{corollary}
\label{cor:boundary_all_extreme}
Suppose that $\mathrm{Aut}(\mathcal{S})$ acts transitively on $\mathcal{S}_{\mathrm{ext}}$.
Then $\mathcal{S}$ is affine isomorphic to the unit ball if and only if every boundary point of $\mathcal{S}$ is an extreme point of $\mathcal{S}$.
\end{corollary}
\begin{proof}
As the \lq\lq only if'' part is trivial, we show the \lq\lq if'' part.
By virtue of Proposition \ref{prop:isometry}, we may assume without loss of generality that $\mathrm{Aut}(\mathcal{S})$ acts on $\mathcal{S}$ as linear isometries (with respect to the Euclidean distance on $V(\mathcal{S})$).
This implies that every boundary point of $\mathcal{S}$, which is now an extreme point of $\mathcal{S}$ by the assumption, lies in a common sphere $S$ in $V(\mathcal{S})$ with the origin as the center point, hence we have $\partial \mathcal{S} = S$ and $\mathcal{S}$ is a ball surrounded by $S$, as desired.
\end{proof}

On the other hand, we present several properties for the notion of distinguishability introduced in Definition \ref{defn:distinguishability}.
First, we give a geometric interpretation of the notion of distinguishability mentioned in Section \ref{subsec:introduction_summary}:
\begin{lemma}
\label{lem:distinguishable_states}
We temporarily relax the assumption that $\mathcal{S}$ is compact.
Let $n \geq 2$.
Then points $s_1,s_2,\dots,s_n$ of $\mathcal{S}$ are distinguishable (see Definition \ref{defn:distinguishability} for the terminology) if and only if there exists a collection $(H_i)_{i=1}^{n}$ of supporting hyperplanes $H_i$ of $\mathcal{S}$ such that the set $\{v_1,\dots,v_n\}$ of normal vectors $v_i$ of $H_i$ is linearly dependent and we have $s_i \not\in H_i$ and $s_i \in H_j$ for every $i \neq j$.
\end{lemma}
\begin{proof}
Let $\langle \cdot,\cdot \rangle$ denote the standard inner product on the Euclidean space $V(\mathcal{S})$.
First we show the \lq\lq only if'' part.
Choose the affine functionals $e_i$ as in Definition \ref{defn:distinguishability}.
For each $i$, there are a non-zero vector $v_i \in V(\mathcal{S})$ and a constant $c_i \in \mathbb{R}$ such that $e_i(s) = \langle v_i,s \rangle + c_i$ for every $s \in \mathcal{S}$.
As $\sum_{i=1}^{n} e_i = 1$ is constant on $\mathcal{S}$ and $\mathrm{Aff}(\mathcal{S}) = V(\mathcal{S})$, it follows that $\sum_{i=1}^{n} v_i = 0$.
Put $H_i = \{v \in V(\mathcal{S}) \mid e_i(v) = 0\}$, $1 \leq i \leq n$, where we abuse the notation $e_i$ to denote the affine extension of $e_i$ to $V(\mathcal{S}) = \mathrm{Aff}(\mathcal{S})$.
Then we have $s_i \not\in H_i$ and $s_i \in H_j$ for every $i \neq j$, therefore $H_i \cap \mathcal{S} \neq \emptyset$.
As $e_i \geq 0$ on $\mathcal{S}$, it follows that the $H_i$ are supporting hyperplanes of $\mathcal{S}$.
Moreover, $v_i$ is proportional to the normal vector of $H_i$, therefore the normal vectors of $H_i$ are linearly dependent.
Hence the proof of the \lq\lq only if'' part is concluded.

Secondly, we show the \lq\lq if'' part.
Choose coefficients $\lambda_1,\dots,\lambda_n \in \mathbb{R}$ such that $\sum_{i=1}^{n} \lambda_i v_i = 0$ and at least one $\lambda_i$ is non-zero.
Choose a point $t_i \in H_i$ for each $i$.
We define an affine functional $e_i$ on $V(\mathcal{S})$ by $e_i(v) = \langle \lambda_i v_i, v - t_i \rangle$ for $v \in V(\mathcal{S})$.
We have $e_i(H_i) = \{0\}$ as $v_i$ is orthogonal to $H_i$.
On the other hand, by putting $e = \sum_{i=1}^{n} e_i$, for each $v \in V(\mathcal{S})$ we have
\begin{equation}
e(v) = \sum_{i=1}^{n} e_i(v)
= \sum_{i=1}^{n} (\langle \lambda_i v_i, v \rangle - \langle \lambda_i v_i, t_i \rangle)
= \langle \sum_{i=1}^{n} \lambda_i v_i, v \rangle - \sum_{i=1}^{n} \langle \lambda_i v_i, t_i \rangle
= - \sum_{i=1}^{n} \langle \lambda_i v_i, t_i \rangle \enspace.
\end{equation}
Put $c = - \sum_{i=1}^{n} \langle \lambda_i v_i, t_i \rangle$, therefore $e$ is constantly equal to $c$.
Now choose an index $i_0$ such that $\lambda_{i_0} \neq 0$.
Then, as $s_{i_0} \in H_j$ and $e_j(H_j) = \{0\}$ for every $j \neq i_0$, we have $e_{i_0}(s_{i_0}) = e(s_{i_0}) = c$.
As $s_{i_0} \not\in H_{i_0}$ and $\lambda_{i_0} \neq 0$, we have $e_{i_0}(s_{i_0}) = \lambda_{i_0} \langle v_{i_0}, s_{i_0} - t_{i_0} \rangle \neq 0$, therefore $c \neq 0$.
Now put $e'_i = e_i / c$ for each $i$.
Then we have $e'_i(H_i) = \{0\}$, $\sum_{i=1}^{n} e'_i = e / c = 1$ and $e'_i(s_i) = e_i(s_i) / c = e(s_i) / c = 1$ for each $i$.
As $H_i$ is a supporting hyperplane of $\mathcal{S}$ and $s_i \in \mathcal{S}$, this implies that $\mathcal{S}$ is included in the closed half-space $\{v \in V(\mathcal{S}) \mid e'_i(v) \geq 0\}$, therefore $e'_i \geq 0$ on $\mathcal{S}$.
Hence $s_1,\dots,s_n$ are distinguishable, concluding the proof of Lemma \ref{lem:distinguishable_states}.
\end{proof}
\begin{remark}
\label{rem:distinguishable_states_hyperplanes}
In the situation of Lemma \ref{lem:distinguishable_states}, any proper subset of the set $\{v_1,\dots,v_n\}$ of normal vectors of the hyperplanes $H_i$ is linearly independent.
Indeed, when we choose $\lambda_1,\dots,\lambda_n \in \mathbb{R}$ as in the proof of that lemma, for each $i$ we have $e_i(s_i) = e(s_i) = c \neq 0$, while $e_i(s_i) = \lambda_i \langle v_i, s_i - t_i \rangle$.
Therefore we have $\lambda_i \neq 0$ for every $i$, which implies the claim.
\end{remark}
Now we give an upper bound of size of a collection of distinguishable points:
\begin{lemma}
\label{lem:maximal_distinguishable}
We temporarily relax the assumption that $\mathcal{S}$ is compact.
Suppose that $s_1,\dots,s_k \in \mathcal{S}$ are distinguishable.
Then we have $k \leq \dim(\mathcal{S}) + 1$.
Moreover, if in addition $k = \dim(\mathcal{S}) + 1$, then $\mathcal{S}$ is the convex hull of $\{s_j\}_{j=1}^{k}$, which is a $\dim(\mathcal{S})$-dimensional simplex.
\end{lemma}
\begin{proof}
Let $(e_j)_{j=1}^{k}$ be the collection of affine functionals corresponding to $(s_j)_{j=1}^{k}$, therefore $e_i(s_j) = \delta_{i,j}$.
Now if $s = \sum_{j=1}^{k} \lambda_j s_j$ and $\sum_{j=1}^{k} \lambda_j = 1$, then we have $e_i(s) = \sum_{j} \lambda_j e_i(s_j) = \lambda_i$ for every $i$, therefore the decomposition of $s$ into the points $s_j$ is uniquely determined.
Hence the points $s_1,\dots,s_k$ are affine independent, therefore $k \leq \dim(\mathcal{S}) + 1$.
Moreover, if $k = \dim(\mathcal{S}) + 1$, then we have $V(\mathcal{S}) = \mathrm{Aff}(s_1,\dots,s_k)$ as the points $s_1,\dots,s_k$ are affine independent, therefore each $s \in \mathcal{S}$ admits a decomposition $s = \sum_{j=1}^{k} \lambda_j s_j$ such that $\sum_{j} \lambda_j = 1$.
Now we have $\lambda_j = e_j(s) \in [0,1]$ for every $j$, therefore this $s$ is contained in the convex hull of $\{s_1,\dots,s_k\}$.
Hence $\mathcal{S}$ is the convex hull of $\{s_1,\dots,s_k\}$, therefore Lemma \ref{lem:maximal_distinguishable} holds.
\end{proof}

Secondly, for the notion of ($n$-)spectrality introduced in Definition \ref{defn:distinguishably_decomposable}, we have the following two properties (we emphasize again that, as mentioned in the introduction, the distinguishable points in the definition of spectrality should be {\em extreme points}, hence e.g., the square is ruled out):
\begin{lemma}
\label{lem:k-DDPS_pure_states}
We temporarily relax the assumption that $\mathcal{S}$ is compact.
Suppose that $1 \leq \dim(\mathcal{S}) = n < \infty$ and $\mathcal{S}$ has $n$-spectrality.
Then the set $\mathcal{S}_{\mathrm{ext}}$ is uncountably infinite.
\end{lemma}
\begin{proof}
Assume for contrary that $\mathcal{S}_{\mathrm{ext}}$ is at most countable.
Let $\mathcal{B}$ be the collection of the subsets of $\mathcal{S}_{\mathrm{ext}}$ with at most $n$ elements, and let $\mathcal{C}$ be the collection of the convex hulls of $B \in \mathcal{B}$.
Then each member of $\mathcal{C}$ has $n$-dimensional volume $0$, while $\mathcal{B}$ (hence $\mathcal{C}$) is at most countable as well as $\mathcal{S}_{\mathrm{ext}}$, therefore the union of all members of $\mathcal{C}$ also has $n$-dimensional volume $0$.
As the $n$-dimensional convex set $\mathcal{S}$ has positive $n$-dimensional volume, there is a point $s \in \mathcal{S}$ that does not belong to any member of $\mathcal{C}$.
However, as $\mathcal{S}$ has $n$-spectrality, this $s$ admits a convex decomposition $s = \sum_{j=1}^{\ell} \lambda_j s_j$ into distinguishable extreme points $s_1,\dots,s_{\ell} \in \mathcal{S}_{\mathrm{ext}}$ with $\ell \leq n$, therefore $B = \{s_j\}_{j=1}^{\ell} \in \mathcal{B}$ and the convex hull of $B$ including $s$ belongs to $\mathcal{C}$.
This is a contradiction, hence Lemma \ref{lem:k-DDPS_pure_states} holds.
\end{proof}
\begin{corollary}
\label{cor:distinguishably_decomposable_simplex}
We temporarily relax the assumption that $\mathcal{S}$ is compact.
Suppose that $\mathcal{S}$ has spectrality, and $\mathcal{S}_{\mathrm{ext}}$ is at most countable.
Then $\mathcal{S}$ is a $\dim(\mathcal{S})$-simplex.
\end{corollary}
\begin{proof}
By Lemma \ref{lem:k-DDPS_pure_states}, $\mathcal{S}$ does not have $n$-spectrality, where $n = \dim(\mathcal{S})$, while $\mathcal{S}$ has spectrality.
This implies that there is an $s \in \mathcal{S}$ that admits a convex decomposition into at least $n+1$ distinguishable extreme points.
Therefore Lemma \ref{lem:maximal_distinguishable} implies that $\mathcal{S}$ is an $n$-simplex.
Hence Corollary \ref{cor:distinguishably_decomposable_simplex} holds.
\end{proof}

\subsection{Proof of Theorem \ref{thm:characterization_finite_dim}}
\label{subsec:characterization_finite_dim}

In this and the following subsections, we give proofs of our main theorems whose statements have been listed in Section \ref{subsec:introduction_summary}.
In this subsection, we prove Theorem \ref{thm:characterization_finite_dim}.
First we present the following lemma:
\begin{lemma}
\label{lem:2-transitive_implies_transitive}
Suppose that the diagonal action of $\mathrm{Aut}(\mathcal{S})$ on the set of pairs $(s_1,s_2)$ of distinguishable extreme points $s_1,s_2 \in \mathcal{S}_{\mathrm{ext}}$ is transitive.
Then $\mathrm{Aut}(\mathcal{S})$ acts transitively on $\mathcal{S}_{\mathrm{ext}}$.
\end{lemma}
\begin{proof}
First we note that, for any extreme point $s$ of $\mathcal{S}$, there is another extreme point $s'$ for which $s$ and $s'$ are distinguishable.
Indeed, for any supporting hyperplane $H$ of $\mathcal{S}$ containing $s$, we can take another supporting hyperplane $H' \neq H$ of $\mathcal{S}$ which is parallel to $H$.
Now any extreme point $s'$ of $\mathcal{S} \cap H'$ is also an extreme point in $\mathcal{S}$, and $s$ and $s'$ are distinguishable by the choice of $H$ and $H'$.
This implies that, for any two $s_1,s_2 \in \mathcal{S}_{\mathrm{ext}}$, there are $s'_1,s'_2 \in \mathcal{S}_{\mathrm{ext}}$ for which both the pair $(s_1,s'_1)$ and the pair $(s_2,s'_2)$ are distinguishable, and the action of some element $g$ of $\mathrm{Aut}(\mathcal{S})$ maps $(s_1,s'_1)$ to $(s_2,s'_2)$ by the assumption.
Hence the action of $g$ maps $s_1$ to $s_2$, therefore $\mathrm{Aut}(\mathcal{S})$ acts transitively on $\mathcal{S}_{\mathrm{ext}}$, as desired.
\end{proof}

For the \lq\lq only if'' part of Theorem \ref{thm:characterization_finite_dim}, when $\mathcal{S}$ is a finite-dimensional unit ball, it follows from Lemma \ref{lem:distinguishable_states} that a subset of $\mathcal{S}$ of cardinality at least two is a set of distinguishable extreme points if and only if it is a pair of mutually antipodal points (note that now each supporting hyperplane of $\mathcal{S}$ contains precisely one point of $\mathcal{S}$).
This implies that $\mathcal{S}$ indeed satisfies the two conditions in Theorem \ref{thm:characterization_finite_dim}, as desired.

On the other hand, for the \lq\lq if'' part of Theorem \ref{thm:characterization_finite_dim}, it suffices by Corollary \ref{cor:boundary_all_extreme} to show that every boundary point $s$ of $\mathcal{S}$ is an extreme point.
Assume for contrary that $s \not\in \mathcal{S}_{\mathrm{ext}}$.
Take a supporting hyperplane $H$ of $\mathcal{S}$ at $s$.
As $\mathcal{S}$ has $2$-spectrality, we have $s = \lambda_1 s_1 + \lambda_2 s_2$ for some distinguishable points $s_1,s_2 \in \mathcal{S}_{\mathrm{ext}}$ and some $\lambda_1,\lambda_2 > 0$ with $\lambda_1 + \lambda_2 = 1$.
On the other hand, by the same reason, the origin $0$ of $V(\mathcal{S})$ (which is now an interior point of $\mathcal{S}$ by Lemma \ref{lem:maximally_mixed_is_interior}) also admits a decomposition $0 = \mu_1 t_1 + \mu_2 t_2$ into distinguishable points $t_1,t_2 \in \mathcal{S}_{\mathrm{ext}}$ with $\mu_1,\mu_2 > 0$ and $\mu_1 + \mu_2 = 1$.
Now we have $s_1,s_2 \in H$ by the choice of $H$ and the decomposition of $s$.
Moreover, by the first condition of Theorem \ref{thm:characterization_finite_dim}, there is an $f \in \mathrm{Aut}(\mathcal{S})$ such that $f(t_1) = s_1$ and $f(t_2) = s_2$.
Now we have $0 = f(0) = \mu_1 f(t_1) + \mu_2 f(t_2) = \mu_1 s_1 + \mu_2 s_2$, while $s_1,s_2 \in H$ as above, therefore we have $0 \in H$.
This contradicts the fact that $0$ is an interior point of $\mathcal{S}$.
Hence we have $s \in \mathcal{S}_{\mathrm{ext}}$ as desired, concluding the proof of Theorem \ref{thm:characterization_finite_dim}.

\subsection{Proofs of Theorems \ref{thm:2-dimensional_strongly_symmetric} and \ref{thm:3-dimensional_strongly_symmetric}}
\label{subsec:characterization_low_dimensional_strongly_symmetric}

In this subsection, we give proofs of Theorems \ref{thm:2-dimensional_strongly_symmetric} and \ref{thm:3-dimensional_strongly_symmetric} by assuming Theorems \ref{thm:2-dimensional_symmetric} and \ref{thm:3-dimensional_symmetric}.
The proofs of Theorems \ref{thm:2-dimensional_symmetric} and \ref{thm:3-dimensional_symmetric} will be supplied in the next subsection.

First, we prove Theorem \ref{thm:2-dimensional_strongly_symmetric}.
The implication of the second condition from the first one follows from Theorem \ref{thm:2-dimensional_symmetric}, Theorem \ref{thm:characterization_finite_dim} and the fact that when $\mathcal{S}$ is a triangle, the set of the three vertices of $\mathcal{S}$ are distinguishable.
For the other implication of the first condition from the second one, by virtue of Theorem \ref{thm:2-dimensional_symmetric}, it suffices to show that any convex polygon having spectrality is a triangle, which is just a consequence of Corollary \ref{cor:distinguishably_decomposable_simplex}.
Hence the proof of Theorem \ref{thm:2-dimensional_strongly_symmetric} is concluded.

From now, we prove Theorem \ref{thm:3-dimensional_strongly_symmetric}.
The implication of the second condition from the first one follows from Theorem \ref{thm:3-dimensional_symmetric}, Theorem \ref{thm:characterization_finite_dim} and the fact that when $\mathcal{S}$ is a tetrahedron, the set of the four vertices of $\mathcal{S}$ are distinguishable.
For the other implication of the first condition from the second one, by virtue of Theorem \ref{thm:3-dimensional_symmetric}, it suffices to show that any $3$-dimensional convex polytope having spectrality is a tetrahedron, and the $3$-dimensional circular cylinder $\mathcal{S}$ specified in the second condition of Theorem \ref{thm:3-dimensional_symmetric} does not have spectrality.
The first part is just a consequence of Corollary \ref{cor:distinguishably_decomposable_simplex}.
For the second part, put $C_z = \{{}^t(x,y,z) \mid x^2 + y^2 = 1\}$ for $z \in \{0,1\}$.
Note that $\mathcal{S}_{\mathrm{ext}} = C_0 \cup C_1$.
We focus on the point $v = {}^t(0,0,1/4)$, and assume for contrary that $v = \sum_{j = 1}^{k} \lambda_j s_j$ for some $s_j \in \mathcal{S}_{\mathrm{ext}}$ and $\lambda_j > 0$ such that $\sum_{j = 1}^{k} \lambda_j = 1$ and $s_1,\dots,s_k$ are distinguishable (hence $s_1,\dots,s_k$ are all different).
We have $k \leq 3$ by Lemma \ref{lem:maximal_distinguishable}, while the case $k \leq 2$ is impossible by the shape of $\mathcal{S}_{\mathrm{ext}} = C_0 \cup C_1$, therefore we have $k = 3$.
Now the set $\{s_1,s_2,s_3\}$ contains at least one point of each $C_z$, $z = 0,1$.
By symmetry, we may assume without loss of generality that $s_1,s_2 \in C_1$ and $s_3 = {}^t(1,0,0) \in C_0$.
As $s_1,s_2,s_3$ are distinguishable, Lemma \ref{lem:distinguishable_states} implies that there are supporting hyperplanes $H_1,H_2$ of $\mathcal{S}$ such that $s_1,s_3 \in H_1$ and $s_2,s_3 \in H_2$.
In particular, for each $j \in \{1,2\}$, the line through $s_j$ and $s_3$ does not intersect the interior of $\mathcal{S}$.
However, by the shape of $\mathcal{S}$, this is possible only when $s_j = {}^t(1,0,1)$, contradicting the fact $s_1 \neq s_2$.
Hence this $\mathcal{S}$ does not have spectrality, concluding the proof of Theorem \ref{thm:3-dimensional_strongly_symmetric}.

\subsection{Proofs of Theorems \ref{thm:2-dimensional_symmetric} and \ref{thm:3-dimensional_symmetric}}
\label{subsec:characterization_low_dimensional_symmetric}

Finally, in this subsection, we give proofs of Theorems \ref{thm:2-dimensional_symmetric} and \ref{thm:3-dimensional_symmetric}.
As the \lq\lq if'' part of each theorem is trivial, we consider the \lq\lq only if'' part from now.

First, for future references, we temporarily consider the case that $\mathcal{S}$ has an arbitrary finite (not necessarily $2$ or $3$) dimension, and $\mathrm{Aut}(\mathcal{S})$ acts transitively on $\mathcal{S}_{\mathrm{ext}}$.
Then, by the arguments in Section \ref{subsec:characterization_preliminary}, we assume without loss of generality that the origin $0$ of the Euclidean space $V(\mathcal{S})$ is an interior point of $\mathcal{S}$ which is the unique fixed point of the $\mathrm{Aut}(\mathcal{S})$-action, and each member of $\mathrm{Aut}(\mathcal{S})$ acts on $V(\mathcal{S})$ as an orthogonal transformation (hence has determinant $\pm 1$).
Now we have the following result, which will be a key ingredient of the proofs of the theorems:
\begin{lemma}
\label{lem:infinite_order_exist}
Under the above setting, assume in addition that $\mathrm{Aut}(\mathcal{S})$ is an infinite group.
Then $\mathrm{Aut}(\mathcal{S})$ has an element of infinite order.
\end{lemma}
In the proof of Lemma \ref{lem:infinite_order_exist}, we use the following two theorems.
The first theorem is a classical affirmative solution for a special case of the general Burnside problem given by Schur \cite{Sch} (see also e.g., \cite[Theorem 9.9]{Lam} for the proof).
Recall that a group $G$ is called {\em periodic} (or {\em torsion}) if every element of $G$ has finite order.
The theorem is the following:
\begin{theorem}
[Schur]
\label{thm:Schur}
Let $K$ be a field of characteristic zero, $1 \leq n < \infty$, and $G$ be a finitely generated periodic subgroup of $GL_n(K)$.
Then $|G| < \infty$.
\end{theorem}
The second theorem was proven by Kargapolov \cite{Kar} and independently by Hall and Kulatilaka \cite{HK}.
Recall that a group $G$ is called {\em locally finite} if every finitely generated subgroup of $G$ has finite order.
The theorem is the following:
\begin{theorem}
[Kargapolov, and Hall--Kulatilaka]
\label{thm:locally_finite_abelian}
Every infinite locally finite group contains an infinite abelian subgroup.
\end{theorem}
\begin{proof}
[Proof of Lemma \ref{lem:infinite_order_exist}]
Assume for contrary that $G = \mathrm{Aut}(\mathcal{S})$ has no elements of infinite order, i.e., $G$ is periodic.
As $G$ acts faithfully on $V(\mathcal{S})$ as a group of orthogonal transformations, $G$ is identified with a subgroup of $O_n(\mathbb{R})$ (hence a subgroup of $GL_n(\mathbb{R})$) where $n = \dim(\mathcal{S}) < \infty$.
Then Theorem \ref{thm:Schur} implies that $G$ is locally finite, therefore Theorem \ref{thm:locally_finite_abelian} implies that $G$ contains an infinite abelian subgroup $H$.
As each orthogonal transformation $f \in H$ is diagonalizable over $\mathbb{C}$, the members of $H$ are simultaneously diagonalizable over $\mathbb{C}$.
Let $v_1,\dots,v_n$ be the common eigenvectors of the members of $H$ in the complexification $V(\mathcal{S})^{\mathbb{C}} = \mathbb{C} \otimes_{\mathbb{R}} V(\mathcal{S})$ of $V(\mathcal{S})$, and let $\alpha_{f,1},\dots,\alpha_{f,n} \in \mathbb{C}$ be the corresponding eigenvalues of $f \in H$.
As each $f \in H$ is an orthogonal transformation, we have $|\alpha_{f,j}| = 1$ for every $f \in H$ and $1 \leq j \leq n$.
Note that $(\alpha_{f,1},\dots,\alpha_{f,n}) \neq (\alpha_{g,1},\dots,\alpha_{g,n})$ for any distinct $f,g \in H$.

First, we show that there exist an index $1 \leq j_0 \leq n$ and an infinite subset $H_0 \subset H$ such that $\alpha_{f,j_0} \not\in \mathbb{R}$ and $\alpha_{f,j_0} \neq \alpha_{g,j_0}$ for any distinct $f,g \in H_0$.
To prove this, we consider the following auxiliary condition, where $d$ is a non-negative integer:
\begin{quote}
There are distinct indices $j_{1,1},j_{1,2},j_{2,1},j_{2,2},\dots,j_{d,1},j_{d,2}$ in $\{1,2,\dots,n\}$ and an infinite subset $H' \subset H$ such that $\alpha_{f,j_{\ell,1}} \not\in \mathbb{R}$, $\alpha_{f,j_{\ell,2}} = \overline{\alpha_{f,j_{\ell,1}}}$ and $\alpha_{f,j_{\ell,1}} = \alpha_{g,j_{\ell,1}}$ for any $1 \leq \ell \leq d$ and any distinct $f,g \in H'$.
\end{quote}
Take the maximal $d \geq 0$ for which this condition holds (hence $d \leq n/2$).
Then there are an index $j_{d+1,1}$ in $\{1,2,\dots,n\}$ other than $j_{1,1},\dots,j_{d,2}$ and an infinite subset $H'_1 \subset H'$ such that $\alpha_{f,j_{d+1,1}} \not\in \mathbb{R}$ for every $f \in H'_1$.
Indeed, otherwise for each index $j$ the variation of the values $\alpha_{f,j}$ for $f \in H'$ is finite (as $|\alpha_{f,j}| = 1$), so is the variation of $(\alpha_{f,1},\dots,\alpha_{f,n})$, contradicting the fact that $(\alpha_{f,1},\dots,\alpha_{f,n}) \neq (\alpha_{g,1},\dots,\alpha_{g,n})$ for any distinct $f,g \in H'$.
Now for each $f \in H'_1$, as $f$ is a real transformation and $\alpha_{f,j_{\ell,2}} = \overline{\alpha_{f,j_{\ell,1}}}$ for every $1 \leq \ell \leq d$, there is an index $j$ in $\{1,2,\dots,n\}$ other than $j_{1,1},\dots,j_{d,2},j_{d+1,1}$ such that $\alpha_{f,j} = \overline{\alpha_{f,j_{d+1,1}}}$.
As $|H'_1| = \infty$, there are an index $j_{d+1,2}$ in $\{1,2,\dots,n\}$ other than $j_{1,1},\dots,j_{d,2},j_{d+1,1}$ and an infinite subset $H'_2 \subset H'_1$ such that $\alpha_{f,j_{d+1,2}} = \overline{\alpha_{f,j_{d+1,1}}}$ for every $f \in H'_2$.
By the maximality of $d$, it follows that for each $f \in H'_2$, there exist at most a finite number of $g \in H'_2$ such that $\alpha_{f,j_{d+1,1}} = \alpha_{g,j_{d+1,1}}$.
This implies that there is an infinite subset $H'_3 \subset H'_2$ such that $\alpha_{f,j_{d+1,1}} \neq \alpha_{g,j_{d+1,1}}$ for any distinct $f,g \in H'_3$.
Now $j_0 = j_{d+1,1}$ and $H_0 = H'_3$ satisfy the desired condition.

In what follows, we write $v = v_{j_0}$ and $\alpha_f = \alpha_{f,j_0}$ for each $f \in H_0$.
Put $\alpha_f = \exp(2 \pi \theta_f \sqrt{-1})$ with $0 < \theta_f < 1$ for each $f \in H_0$ (note that $\alpha_f \not\in \mathbb{R}$).
Then $\overline{v} \in V(\mathcal{S})^{\mathbb{C}}$ is an eigenvector of each $f \in H_0$ with eigenvalue $\overline{\alpha_f}$.
By putting $u = (v - \overline{v}) / (2 \sqrt{-1})$ and $w = (v + \overline{v}) / 2$, it follows that $u,w \in V(\mathcal{S})$, $v = w + \sqrt{-1} u$, $f(u) = \cos(2 \pi \theta_f) u + \sin(2 \pi \theta_f) w$ and $f(w) = -\sin(2 \pi \theta_f) u + \cos(2 \pi \theta_f) w$ for each $f \in H_0$.
Note that $u$ and $w$ are linearly independent over $\mathbb{R}$, as $v$ and $\overline{v}$ are linearly independent over $\mathbb{C}$.
As $0 \in \mathrm{int}_{V(\mathcal{S})}(\mathcal{S})$, by considering suitable scalar multiplication if necessary, we may assume without loss of generality that $u,w \in \mathcal{S}$.
On the other hand, as each $f \in H_0$ has finite order by the assumption, we have $\theta_f \in \mathbb{Q}$ and we can write $\theta_f = p_f / q_f$ with $p_f$ and $q_f$ being coprime integers such that $1 \leq p_f < q_f$.
Then, as $\alpha_{f,1} \neq \alpha_{g,1}$ for any distinct $f,g \in H_0$ and $|H_0| = \infty$, for each $N > 0$ there is an $f \in H_0$ such that $q_f > N$.

Now choose an irrational $0 < \eta < 1$ (say, $\eta = 1/\sqrt{2}$).
Put $u_{\infty} = \cos(2 \pi \eta) u + \sin(2 \pi \eta) w $ and $w_{\infty} = -\sin(2 \pi \eta) u + \cos(2 \pi \eta) w$.
Then for each integer $k \geq 1$, there is an $f_k \in H_0$ such that $q_{f_k} > k$ as above.
Now there is an integer $r_k$ such that $|\eta - r_k / q_{f_k}| < 1/(2k)$.
As $p_{f_k}$ and $q_{f_k}$ are coprime, there is an integer $h_k$ such that $\exp(2\pi h_k\theta_{f_k}\sqrt{-1}) = \exp(2\pi r_k \sqrt{-1} / q_{f_k})$.
Put $g_k = f_k{}^{h_k} \in G$.
Then we have
\begin{equation}
\begin{split}
g_k(u) &= \cos(2 \pi h_k \theta_{f_k}) u + \sin(2 \pi h_k \theta_{f_k}) w
= \cos(2 \pi r_k / q_{f_k}) u + \sin(2 \pi r_k / q_{f_k}) w \enspace, \\
g_k(w) &= -\sin(2 \pi h_k \theta_{f_k}) u + \cos(2 \pi h_k \theta_{f_k}) w
= -\sin(2 \pi r_k / q_{f_k}) u + \cos(2 \pi r_k / q_{f_k}) w \enspace.
\end{split}
\end{equation}
As $r_k / q_{f_k}$ converges to $\eta$ when $k \to \infty$, the sequences $(g_k(u))_{k \geq 1}$ and $(g_k(w))_{k \geq 1}$ converges to $u_{\infty}$ and $w_{\infty}$, respectively.
On the other hand, it follows from a general result that the compact-open topology on $\mathsf{Aut}(\mathcal{S})$ is metrizable (see e.g., \cite[Section c-20.5]{top_book}), therefore $G$ is a compact metric space.
This implies that the sequence $(g_k)_{k \geq 1}$ in $G$ has a subsequence $(g_{k_{\ell}})_{\ell \geq 1}$ that converges to some $g \in G$.
As the $G$-action on $\mathcal{S}$ is continuous and $u,w \in \mathcal{S}$, it follows that $g(u) = \lim_{\ell \to \infty} g_{k_{\ell}}(u) = u_{\infty}$ and $g(w) = \lim_{\ell \to \infty} g_{k_{\ell}}(w) = w_{\infty}$.
This implies that $g^k(u) = \cos(2 \pi k \eta) u + \sin(2 \pi k \eta) w$ and $g^k(w) = -\sin(2 \pi k \eta) u + \cos(2 \pi k \eta) w$ for every $k \geq 1$, therefore $g^k(u) \neq u$ as $u$ and $w$ are linearly independent over $\mathbb{R}$ and $\eta \not\in \mathbb{Q}$.
Hence $g$ has infinite order, concluding the proof of Lemma \ref{lem:infinite_order_exist}.
\end{proof}
If $\mathcal{S}_{\mathrm{ext}}$ is a finite set, then $\mathcal{S}$ is a convex polytope, which is symmetric (or vertex-transitive) by the assumption on the $\mathrm{Aut}(\mathcal{S})$-action.
From now, we suppose that $\mathcal{S}_{\mathrm{ext}}$ is an infinite set, hence $\mathrm{Aut}(\mathcal{S})$ is also infinite as it acts transitively on $\mathcal{S}_{\mathrm{ext}}$.
Let $f \in \mathrm{Aut}(\mathcal{S})$ be an element of infinite order whose existence is proven by Lemma \ref{lem:infinite_order_exist}.
We may assume without loss of generality that $\det(f) = 1$, by considering (if necessary) $f^2$ instead of $f$.
Now as $f$ is an orthogonal transformation, it follows that $V(\mathcal{S})$ is the orthogonal direct sum of the fixed point space of $f$ and $2$-dimensional $f$-invariant subspaces $W$ to which the restriction of $f$ is a rotation with rotation angle $2 \pi \theta = 2 \pi \theta_W$, $\theta_W \in \mathbb{R}$.
Moreover, as $f$ has infinite order, $\theta = \theta_W$ is irrational for at least one such subspace $W$.

Now we come back to the special case for (the \lq\lq only if'' parts of) Theorems \ref{thm:2-dimensional_symmetric} and \ref{thm:3-dimensional_symmetric} by assuming that $n = \dim(\mathcal{S}) \in \{2,3\}$.
First we consider the case that $n = 2$ (for Theorem \ref{thm:2-dimensional_symmetric}).
Then we have $W = V(\mathcal{S})$, while the rotation angle of $f$ is not a rational multiple of $2 \pi$ as above, therefore the $\langle f \rangle$-orbit of any extreme point of $\mathcal{S}$ is a dense subset of a circle with the origin of $V(\mathcal{S})$ as center point.
As $\mathcal{S}_{\mathrm{ext}}$ is compact (see Proposition \ref{prop:pure_states_symmetric_GPT}), this circle is entirely contained in $\mathcal{S}_{\mathrm{ext}}$, which implies that $\mathcal{S}$ is the disk surrounded by this circle.
Hence the proof of Theorem \ref{thm:2-dimensional_symmetric} is concluded.

From now, we consider the case that $n = 3$ (for Theorem \ref{thm:3-dimensional_symmetric}).
In this case, by counting the dimension, it follows that $V(\mathcal{S})$ is the orthogonal direct sum of $1$-dimensional fixed point space of $f$ and the $2$-dimensional invariant space $W$ as above.
By choosing a coordinate system of $V(\mathcal{S}) = \mathbb{R}^3$ according to the decomposition, we assume without loss of generality that $W$ has orthonormal basis $\{v_1 = {}^t(1,0,0), v_2 = {}^t(0,1,0)\}$ and $v_3 = {}^t(0,0,1)$ is fixed by $f$.
Let $z_{\min}$ and $z_{\max}$ be the minimal and the maximal values $z_0 \in \mathbb{R}$, respectively, such that the plane $H_{z_0} = \{{}^t(x,y,z) \mid z = z_0\}$ has non-empty intersection with $\mathcal{S}$.
Note that $z_{\min} < 0 < z_{\max}$, as $0 \in \mathrm{int}_{V(\mathcal{S})}(\mathcal{S})$.
Note also that each $H_{z_0}$ is invariant under $f$.
Now define $I = \{z \in [z_{\min},z_{\max}] \mid H_z \cap \mathcal{S}_{\mathrm{ext}} \neq \emptyset\}$.
Note that $z_{\min},z_{\max} \in I$.
Then for each $z_0 \in I$, as $f(H_{z_0}) = H_{z_0}$, the same argument as the case $n = 2$ implies that $H_{z_0} \cap \mathcal{S}_{\mathrm{ext}} = \{{}^t(x,y,z_0) \mid x^2 + y^2 = r_{z_0}\}$ and $H_{z_0} \cap \mathcal{S} = \{{}^t(x,y,z_0) \mid x^2 + y^2 \leq r_{z_0}\}$ for some $r_{z_0} \geq 0$.
Note that $r_z > 0$ for every $z \in I \setminus \{z_{\min},z_{\max}\}$.
We divide the rest of the proof into two cases.

First, we consider the case that $I \neq \{z_{\min},z_{\max}\}$.
Then by the above argument, there are a plane $H$ and a simple closed curve $C$ such that $H$ has non-empty intersection with the interior of $\mathcal{S}$ and $C \subset H \cap \mathcal{S}_{\mathrm{ext}}$.
Put $H' = H_{z_{\max}}$, which is a supporting hyperplane of $\mathcal{S}$.
Then, as $\mathrm{Aut}(\mathcal{S})$ acts transitively on $\mathcal{S}_{\mathrm{ext}}$, there is a $g \in \mathrm{Aut}(\mathcal{S})$ such that $g(C) \subset \mathcal{S}_{\mathrm{ext}}$ has a point $s$ common to $H' \cap \mathcal{S}_{\mathrm{ext}}$.
Now as $g(H)$ has non-empty intersection with the interior of $\mathcal{S}$ as well as $H$, we have $g(H) \neq H'$, therefore $g(C) \setminus H'$ has a point $s'$.
This implies that $\mathcal{S}_{\mathrm{ext}}$ contains a simple curve joining $s \in H'$ and $s' \not\in H'$, therefore we have $s' \in H_{z_0}$ and $[z_0,z_{\max}] \subset I$ for some $z_0 < z_{\max}$.
Moreover, by the convexity of $\mathcal{S}$ and the above description of the set $H_z \cap \mathcal{S}_{\mathrm{ext}}$ for $z \in I$, this implies further that $H_z \cap \partial \mathcal{S} = H_z \cap \mathcal{S}_{\mathrm{ext}}$ for every $z \in (z_0,z_{\max})$.
Therefore any point of $H_{z'} \cap \mathcal{S}_{\mathrm{ext}}$, where $z' = (z_0 + z_{\max})/2$, belongs to the interior of $\mathcal{S}_{\mathrm{ext}}$ relative to $\partial \mathcal{S}$.
As the $\mathrm{Aut}(\mathcal{S})$-action is transitive on $\mathcal{S}_{\mathrm{ext}}$ and continuous on $\partial \mathcal{S}$, it follows that every point of $\mathcal{S}_{\mathrm{ext}}$ belongs to the interior of $\mathcal{S}_{\mathrm{ext}}$ relative to $\partial \mathcal{S}$.
This implies that $\mathcal{S}_{\mathrm{ext}} = \partial \mathcal{S}$, as $\partial \mathcal{S}$ is connected and $\mathcal{S}_{\mathrm{ext}}$ is compact.
Hence Corollary \ref{cor:boundary_all_extreme} implies that it is in the third case of Theorem \ref{thm:3-dimensional_symmetric}, as desired.

Secondly, we consider the case that $I = \{z_{\min},z_{\max}\}$, therefore $\mathcal{S}_{\mathrm{ext}} \subset H_{z_{\min}} \cup H_{z_{\max}}$.
Note that $r_{z_{\min}} > 0$ or $r_{z_{\max}} > 0$, as $0 \in \mathrm{int}_{V(\mathcal{S})}(\mathcal{S})$.
Now by the above description of the set $H_z \cap \mathcal{S}_{\mathrm{ext}}$ for $z \in I$, it is in the second case of Theorem \ref{thm:3-dimensional_symmetric} if $r_{z_{\min}} = r_{z_{\max}}$.
From now, we assume for contrary that $r_{z_{\min}} \neq r_{z_{\max}}$ and deduce a contradiction.
By symmetry, we may assume without loss of generality that $r_{z_{\min}} > r_{z_{\max}}$.
Moreover, by choosing a suitable coordinate system, we may assume without loss of generality that $z_{\min} = -1$ and $r_{z_{\min}} = 1$.
Put $z_{\max} = z > 0$ and $r_{z_{\max}} = r < 1$.
As $\mathrm{Aut}(\mathcal{S})$ acts transitively on $\mathcal{S}_{\mathrm{ext}}$, there is a $g \in \mathrm{Aut}(\mathcal{S})$ such that $g({}^t(1,0,-1)) = {}^t(r,0,z)$.
Now $g$ permutes the two connected components $H_{-1} \cap \mathcal{S}_{\mathrm{ext}}$ and $H_z \cap \mathcal{S}_{\mathrm{ext}}$ of $\mathcal{S}_{\mathrm{ext}}$.
This implies that $r > 0$ and $g(H_z) = H_{-1}$, as $H_z$ and $H_{-1}$ are affine hulls of $H_z \cap \mathcal{S}_{\mathrm{ext}}$ and $H_{-1} \cap \mathcal{S}_{\mathrm{ext}}$, respectively.
Put $v = {}^t(-z,0,z)$.
Then we have $v \in H_z$ and $v = -z \cdot {}^t(1,0,-1)$, therefore $g(v) \in H_{-1}$ and $g(v) = -z \cdot g({}^t(1,0,-1)) = {}^t(-r z,0,-z^2)$.
Hence we have $-z^2 = -1$, therefore $z = 1$ as $z > 0$.
Moreover, as $r < 1$, we have $v = {}^t(-1,0,1) \in H_1 \setminus \mathcal{S}$ and $g(v) = {}^t(-r,0,-1) \in \mathcal{S}$, a contradiction.
Hence the proof of Theorem \ref{thm:3-dimensional_symmetric} is concluded.

\section{Preliminaries for general setting}
\label{sec:preliminary}

In the rest of the paper, we study some fundamental properties towards generalizations of the existing operational arguments on general physical theories, including derivation of quantum theory such as in the previous section, from the case of compact and finite-dimensional physical systems to more general cases.
As mentioned in the introduction, such a generalized argument would be worthy from the viewpoint of removing some possibly unnecessary technical assumptions from the study of general physical theories.
This section is devoted to preliminaries for more detailed arguments in Section \ref{sec:dynamics_topology}.

For fundamental facts and terminology regarding topological spaces, we refer to \cite{top_book} or Prerequisites in \cite{tvs}.
We refer to \cite{tvs} for fundamentals of topological vector spaces.
First, we introduce the following terminology, which has been used in the physically motivated preceding works on convex structures mentioned in Section \ref{subsec:introduction_summary}:
\begin{definition}
\label{defn:separated}
We say that a convex set $\mathcal{S}$ is {\em separated} if for any distinct elements $s_1,s_2 \in \mathcal{S}$, there exists an affine functional $f:\mathcal{S} \to \mathbb{R}$ such that $f(\mathcal{S})$ is bounded in $\mathbb{R}$ and $f(s_1) \neq f(s_2)$.
\end{definition}
We notice that, when $\mathcal{S}$ is a convex subset of a finite-dimensional Euclidean space, $\mathcal{S}$ is separated if and only if $\mathcal{S}$ is bounded (see e.g., \cite[Lemma 2.5.1 and Section 2.5, Exercise 1]{Gru}).
Now according to a standard argument (see e.g., \cite{Hol_book}), for our purpose we may assume without loss of generality that a convex set $\mathcal{S}$ is always separated (we describe the argument in \ref{sec:appendix_justification_separated_condition} for the sake of completeness).
In the rest of the paper, we suppose without loss of generality that any convex set denoted by the symbols $\mathcal{S}$, $\mathcal{S}_0$, $\mathcal{S}_1$, $\mathcal{S}_2,\dots$ is separated unless otherwise specified.
Now we present the following preceding result as the starting point of our argument, which says intuitively that any separated convex set has compact convex closure in the completion of the underlying vector space:
\begin{proposition}
[{\cite[Theorem 2.1]{NKM}}]
\label{prop:minimal_framework}
For any (separated) convex set $\mathcal{S}$, there exists a unique (up to isomorphism) collection $(\widetilde{\mathcal{S}},V(\mathcal{S}),\widetilde{V}(\mathcal{S}))$ of objects with the following properties:
\begin{itemize}
\item $V(\mathcal{S})$ and $\widetilde{V}(\mathcal{S})$ are locally convex Hausdorff topological vector spaces such that $V(\mathcal{S})$ is a dense subspace of $\widetilde{V}(\mathcal{S})$.
\item $\mathcal{S}$ is a convex subset of $V(\mathcal{S})$ such that $\mathrm{Aff}(\mathcal{S}) = V(\mathcal{S})$.
\item Let $\mathcal{L}$ denote the set of all continuous linear functionals on $\widetilde{V}(\mathcal{S})$.
Then the weak topology on $\widetilde{V}(\mathcal{S})$ induced by the set $\mathcal{L}$ of mappings coincides with the original topology of $\widetilde{V}(\mathcal{S})$.
\item The weak topology on $V(\mathcal{S})$, induced by the set of all linear functionals $f$ on $V(\mathcal{S})$ such that $f(\mathcal{S}) \subset \mathbb{R}$ is bounded, coincides with the original topology of $V(\mathcal{S})$.
\item The weak topology on $\mathcal{S}$, induced by the set of all affine functions $f:\mathcal{S} \to [0,1]$, coincides with the original topology of $\mathcal{S}$ (the subspace topology relative to $V(\mathcal{S})$).
\item $\widetilde{\mathcal{S}} = \mathrm{cl}_{\widetilde{V}(\mathcal{S})}(\mathcal{S})$, $\widetilde{\mathcal{S}}$ is convex, compact and complete, and $\mathrm{Aff}(\widetilde{\mathcal{S}}) = \widetilde{V}(\mathcal{S})$.
\end{itemize}
\end{proposition}
In what follows, we suppose that associated to each convex set $\mathcal{S}$ the objects $\widetilde{\mathcal{S}}$, $V(\mathcal{S})$ and $\widetilde{V}(\mathcal{S})$ as in Proposition \ref{prop:minimal_framework} and the induced topology on $\mathcal{S}$ are given.
Note that $V(\mathcal{S}) = \widetilde{V}(\mathcal{S})$ and $\widetilde{\mathcal{S}} = \mathcal{S}$ if $\mathcal{S}$ is compact.
On the other hand, if $\mathcal{S}$ is finite-dimensional, then $V(\mathcal{S}) = \widetilde{V}(\mathcal{S})$ and it is a Euclidean space of the same dimension as $\mathcal{S}$.

We present two lemmas for later reference.
The first one is the following:
\begin{lemma}
\label{lem:extends_to_affine}
Any continuous affine map $f:\mathcal{S}_1 \to \widetilde{\mathcal{S}}_2$ extends uniquely to an affine map $\widetilde{f}:\widetilde{\mathcal{S}}_1 \to \widetilde{\mathcal{S}}_2$, hence $\widetilde{f}|_{\mathcal{S}_1} = f$.
\end{lemma}
\begin{proof}
The uniqueness follows from the fact that $\widetilde{\mathcal{S}}_1 = \mathrm{cl}_{\widetilde{\mathcal{S}}_1}(\mathcal{S}_1) \subset \widetilde{V}(\mathcal{S})$ is Hausdorff.
For the existence, as $\mathrm{Aff}(\mathcal{S}_1) = V(\mathcal{S}_1)$, $f$ extends to an affine map $g:V(\mathcal{S}_1) \to \widetilde{V}(\mathcal{S}_2)$.
Choose $v \in \widetilde{V}(\mathcal{S}_2)$ such that $g + v$ is a linear map.
Then $g + v$ is continuous at a point in $V(\mathcal{S}_1)$, namely any point in $\mathcal{S}_1$, therefore it is uniformly continuous on $V(\mathcal{S}_1)$.
As $\widetilde{V}(\mathcal{S}_1)$ contains $V(\mathcal{S}_1)$ as a dense subspace, the map $g + v$ extends to a unique continuous linear map $h:\widetilde{V}(\mathcal{S}_1) \to \widetilde{V}(\mathcal{S}_2)$ (see e.g., \cite[Section III.1]{tvs}).
Now $\overline{f} = h - v:\widetilde{V}(\mathcal{S}_1) \to \widetilde{V}(\mathcal{S}_2)$ is a continuous affine extension of $f$.
This implies that $\overline{f}(\mathcal{S}_1) \subset \widetilde{\mathcal{S}}_2$, therefore $\overline{f}(\widetilde{\mathcal{S}}_1) \subset \widetilde{\mathcal{S}}_2$ as $\widetilde{\mathcal{S}}_1 = \mathrm{cl}_{\widetilde{V}(\mathcal{S}_1)}(\mathcal{S}_1)$, $\widetilde{\mathcal{S}}_2$ is closed and $\overline{f}$ is continuous.
Hence we obtain a continuous affine extension $\widetilde{f} = \overline{f}|_{\widetilde{\mathcal{S}}_1}:\widetilde{\mathcal{S}}_1 \to \widetilde{\mathcal{S}}_2$ of $f$.
\end{proof}
The second lemma below says intuitively that a closed subset $C$ of an open set $U$ will be still contained in $U$ after a slight moving toward any point:
\begin{lemma}
\label{lem:perturb_closed_set}
Let $C$ and $U$ be a closed subset and an open subset of $\widetilde{\mathcal{S}}$, respectively, such that $C \subset U$.
Then for any $x \in \widetilde{\mathcal{S}}$, there exists an $m \in (0,1)$ such that $\lambda x + (1-\lambda) y \in U$ for every $y \in C$ and $0 \leq \lambda \leq m$.
\end{lemma}
\begin{proof}
First, as $\widetilde{\mathcal{S}}$ is compact and Hausdorff, $\widetilde{\mathcal{S}}$ is a normal space and $C$ is compact.
The Urysohn's Lemma implies that there exists a continuous map $F:\widetilde{\mathcal{S}} \to [0,1]$ such that $C \subset F^{-1}(0)$ and $U^c \subset F^{-1}(1)$.
Then the map $\varphi:[0,1] \times C \to [0,1]$, $\varphi(\lambda,y) = F(\lambda x + (1-\lambda) y)$, is also continuous, as the operation of taking a convex combination of two elements in $\widetilde{\mathcal{S}}$ is continuous.
Note that $\{0\} \times C \subset \varphi^{-1}(0) \subset \varphi^{-1}([0,1))$, therefore for each $y \in C$, there are relatively open neighborhoods $I_y \subset [0,1]$ of $0$ and $W_y \subset C$ of $y$, respectively, such that $I_y \times W_y \subset \varphi^{-1}([0,1))$.
As $C$ is compact, there are finitely many elements $y_1,\dots,y_k \in C$ such that $C = \bigcup_{i=1}^{k} W_{y_i}$.
Now $\bigcap_{i=1}^{k} I_y$ is a relatively open neighborhood of $0$ in $[0,1]$, therefore $[0,m] \subset \bigcap_{i=1}^{k} I_y$ for some $0 < m < 1$.
We show that this $m$ satisfies the condition.
Let $y \in C$ and $0 \leq \lambda \leq m$.
Then we have $y \in W_{y_i}$ for some $1 \leq i \leq k$.
Now $(\lambda,y) \in I_{y_i} \times W_{y_i}$ and $\varphi(\lambda,y) = F(\lambda x + (1-\lambda) y) < 1$, therefore $\lambda x + (1-\lambda) y \in U$.
Hence Lemma \ref{lem:perturb_closed_set} holds.
\end{proof}
For two convex sets $\mathcal{S}_1$ and $\mathcal{S}_2$, let $\mathcal{A}(\mathcal{S}_1,\mathcal{S}_2)$ denote the set of all affine maps $f:\mathcal{S}_1 \to \mathcal{S}_2$.
Then by Lemma \ref{lem:dynamics_virtual_states} below, $\mathcal{A}(\mathcal{S}_1,\mathcal{S}_2)$ can be embedded into the set $\mathcal{A}^c(\widetilde{A}_1,\widetilde{A}_2)$ of all continuous affine maps $f:\widetilde{A}_1 \to \widetilde{A}_2$ between {\em compact} sets, which would make the situation easier.
However, if we endow the set $\mathcal{A}(\mathcal{S}_1,\mathcal{S}_2)$ with the subspace topology of the standard compact-open topology on $\mathcal{A}^c(\widetilde{A}_1,\widetilde{A}_2)$, then to determine the open subsets of $\mathcal{A}(\mathcal{S}_1,\mathcal{S}_2)$ we need to consider the behavior of (the extension of) a map $f \in \mathcal{A}(\mathcal{S}_1,\mathcal{S}_2)$ at a subset of $\widetilde{\mathcal{S}}_1$ which may be even entirely outside the original set $\mathcal{S}_1$.
To reduce such difficulty, here we define the notion of \lq\lq essential'' open or closed sets as follows, which means intuitively that the \lq\lq essential shape'' of such an essential open or closed set is not affected by pasting or detaching, respectively, the \lq\lq skin'' $\widetilde{\mathcal{S}} \setminus \mathcal{S}$ of the convex set $\mathcal{S}$.
Formally, we present the following definition:
\begin{definition}
\label{defn:essential_subset}
We say that an open subset $O$ of $\widetilde{\mathcal{S}}$ is {\em essential} if $\mathrm{int}_{\widetilde{\mathcal{S}}}(O \cup (\widetilde{\mathcal{S}} \setminus \mathcal{S})) = O$.
On the other hand, we say that a closed subset $C$ of $\widetilde{\mathcal{S}}$ is {\em essential} if $\mathrm{cl}_{\widetilde{\mathcal{S}}}(C \cap \mathcal{S}) = C$.
\end{definition}
\begin{example}
\label{ex:essential_subset}
We consider the case that $\mathcal{S} = \{(x,y) \in \mathbb{R}^2 \mid |x| < 1, |y| < 1\}$.
In this case, we have $\widetilde{\mathcal{S}} = \{(x,y) \in \mathbb{R}^2 \mid |x| \leq 1, |y| \leq 1\}$ and $V(\mathcal{S}) = \widetilde{V}(\mathcal{S}) = \mathbb{R}^2$.
Now $\{(x,y) \in \widetilde{\mathcal{S}} \mid y < x\}$ and $\{(x,y) \in \mathbb{R}^2 \mid x^2 + y^2 < 1\}$ are essential open subsets of $\widetilde{S}$, while $\mathcal{S}$ is an open subset of $\widetilde{\mathcal{S}}$ but is not essential (note that the interior and the closure in the definition of essential subsets are relative to $\widetilde{\mathcal{S}}$ rather than the underlying vector space $\widetilde{V}(\mathcal{S})$).
\end{example}

Here we show some basic properties of the essential subsets:
\begin{lemma}
\label{lem:essential_subset}
\begin{enumerate}
\item \label{item:essential_subset_1}
A subset $O$ of $\widetilde{\mathcal{S}}$ is open and essential if and only if $O^c \subset \widetilde{\mathcal{S}}$ is closed and essential.
\item \label{item:essential_subset_2}
For any $K \subset \mathcal{S}$ which is closed in $\mathcal{S}$, the set $C = \mathrm{cl}_{\widetilde{\mathcal{S}}}(K)$ is an essential closed subset of $\widetilde{\mathcal{S}}$ and $K = C \cap \mathcal{S}$.
\item \label{item:essential_subset_3}
For any $U \subset \mathcal{S}$ which is open in $\mathcal{S}$, the set $O = \mathrm{int}_{\widetilde{\mathcal{S}}}(U \cup \mathcal{S}^c)$ is an essential open subset of $\widetilde{\mathcal{S}}$ and $U = O \cap \mathcal{S}$.
\item \label{item:essential_subset_4}
For a closed subset $C$ of $\widetilde{\mathcal{S}}$, the following conditions are equivalent:
\begin{enumerate}
\item \label{item:essential_subset_4_a}
$C$ is essential.
\item \label{item:essential_subset_4_b}
$C = \mathrm{cl}_{\widetilde{S}}(K)$ for some $K \subset \mathcal{S}$ which is closed in $\mathcal{S}$.
\item \label{item:essential_subset_4_c}
$C$ is the intersection of all closed subsets $K$ of $\widetilde{\mathcal{S}}$ such that $K \cap \mathcal{S} = C \cap \mathcal{S}$.
\end{enumerate}
\item \label{item:essential_subset_5}
For an open subset $O$ of $\widetilde{\mathcal{S}}$, the following conditions are equivalent:
\begin{enumerate}
\item \label{item:essential_subset_5_a}
$O$ is essential.
\item \label{item:essential_subset_5_b}
$O = \mathrm{int}_{\widetilde{S}}(U \cup \mathcal{S}^c)$ for some $U \subset \mathcal{S}$ which is open in $\mathcal{S}$.
\item \label{item:essential_subset_5_c}
$O$ is the union of all open subsets $U$ of $\widetilde{\mathcal{S}}$ such that $U \cup \mathcal{S}^c = O \cup \mathcal{S}^c$.
\end{enumerate}
\end{enumerate}
\end{lemma}
\begin{proof}
In the proof, we regard operations $A^o$, $\overline{A}$ and $A^c$ as being relative to $\widetilde{\mathcal{S}}$.

(\ref{item:essential_subset_1})
The claim follows from the relation $((O \cup \mathcal{S}^c)^o)^c = \overline{(O \cup \mathcal{S}^c)^c} = \overline{O^c \cap \mathcal{S}}$.

(\ref{item:essential_subset_2})
Choose a closed subset $K'$ of $\widetilde{\mathcal{S}}$ such that $K = K' \cap \mathcal{S}$.
Then we have $K \subset \overline{K} \cap \mathcal{S} \subset (\overline{K'} \cap \overline{\mathcal{S}}) \cap \mathcal{S} = K' \cap \mathcal{S} = K$, therefore $K = \overline{K} \cap \mathcal{S}$ and $\overline{K} = \overline{\overline{K} \cap \mathcal{S}}$.
Hence the claim holds.

(\ref{item:essential_subset_3})
The fact that $O$ is essential follows from the claims \ref{item:essential_subset_1} and \ref{item:essential_subset_2}.
Choose an open subset $U'$ of $\widetilde{\mathcal{S}}$ such that $U = U' \cap \mathcal{S}$.
Then we have $O \cap \mathcal{S} \subset (U \cap \mathcal{S}^c) \cap \mathcal{S} = U = U' \cap \mathcal{S}$, while $U' \subset (U' \cup \mathcal{S}^c)^o = ((U' \cap \mathcal{S}) \cup \mathcal{S}^c)^o = (U \cup \mathcal{S}^c)^o = O$, therefore $O \cap \mathcal{S} \subset U = U' \cap \mathcal{S} \subset O \cap \mathcal{S}$.
Hence we have $O \cap \mathcal{S} = U$, therefore the claim holds.

(\ref{item:essential_subset_4})
The conditions (a) and (b) are equivalent by the definition and the claim \ref{item:essential_subset_2}.
For the implication (a),(b) $\Rightarrow$ (c), $C = \overline{K}$ satisfies that if $K'$ is closed in $\widetilde{\mathcal{S}}$ and $K' \cap \mathcal{S} = C \cap \mathcal{S}$, then $C \cap \mathcal{S} \subset K'$ and $C = \overline{C \cap \mathcal{S}} \subset K'$.
Hence this implication holds.
For the remaining implication (c) $\Rightarrow$ (a), the set $K = \overline{C \cap \mathcal{S}}$ satisfies $K \cap \mathcal{S} = C \cap \mathcal{S}$ by the claim \ref{item:essential_subset_2}, therefore $C \subset \overline{C \cap \mathcal{S}}$.
As $C \cap \mathcal{S} \subset C$ and $C$ is closed, we have $\overline{C \cap \mathcal{S}} = C$, therefore this implication holds.
Hence the three conditions are equivalent.

(\ref{item:essential_subset_5})
By virtue of the claim \ref{item:essential_subset_1}, the claim is derived from the claim \ref{item:essential_subset_4} by taking the complement (relative to $\widetilde{\mathcal{S}}$) of the sets appearing in each statement (note that for the condition (b), we have $C^c = \overline{K}^c = (K^c)^o = (\mathcal{S}^c \cup (\mathcal{S} \setminus K))^o$, and $U = \mathcal{S} \setminus K$ is open in $\mathcal{S}$).
\end{proof}
We also give the following two auxiliary results for later use:
\begin{lemma}
\label{lem:replace_with_essential}
Let $O$ and $C$ be an open and a closed subsets of $\widetilde{\mathcal{S}}$, respectively, such that $O \subset C$.
Then there are an essential open subset $O'$ and an essential closed subset $C'$ of $\widetilde{\mathcal{S}}$, respectively, such that $O \subset O' \subset C' \subset C$.
Moreover, this $C'$ can be chosen to be convex if $C$ is convex.
\end{lemma}
\begin{proof}
First, put $C' = \overline{C \cap \mathcal{S}}$, which is an essential closed subset (by Lemma \ref{lem:essential_subset}) and satisfies that $C' \subset C$, as $C \cap \mathcal{S} \subset C$ and $C$ is closed.
If $C$ is convex, then $C \cap \mathcal{S}$ is also convex, therefore its closure $C'$ is also convex (see e.g., \cite[Section II.1, Theorem 1.2]{tvs}).
We show that $O \subset C'$.
Let $x \in O$ and $W$ an open neighborhood of $x$.
Put $W' = W \cap O$, which is also an open neighborhood of $x$.
Then we have $W' \subset O \subset C$, while $W' \cap \mathcal{S} \neq \emptyset$ as $\mathcal{S}$ is dense in $\widetilde{\mathcal{S}}$.
This implies that $\emptyset \neq W' \cap \mathcal{S} = W' \cap C \cap \mathcal{S} \subset W \cap C \cap \mathcal{S}$, therefore $W \cap (C \cap \mathcal{S}) \neq \emptyset$.
Hence we have $x \in \overline{C \cap \mathcal{S}} = C'$, therefore $O \subset C'$ as desired.
For the remaining claim, by applying the above argument to the pair $C'{}^c \subset O^c$, we have $C'{}^c \subset K \subset O^c$ for some essential closed subset $K$, therefore $O \subset K^c \subset C'$ and $O' = K^c$ is the desired essential open subset by Lemma \ref{lem:essential_subset}.
\end{proof}
\begin{lemma}
\label{lem:closed_convex_neighborhood}
Let $\widehat{\mathcal{S}} \in \{\mathcal{S},\widetilde{\mathcal{S}}\}$.
Then for any $s \in \widehat{\mathcal{S}}$ and any open neighborhood $U$ of $s$ in $\widehat{\mathcal{S}}$, there is a convex closed subset $C$ of $\widehat{\mathcal{S}}$ and a convex open subset $O$ of $\widehat{\mathcal{S}}$ such that $s \in O \subset C \subset U$.
Moreover, in the case $\widehat{\mathcal{S}} = \widetilde{\mathcal{S}}$, this $C$ can be chosen to be essential, and there is an essential open subset $O'$ of $\widetilde{\mathcal{S}}$ such that $O \subset O' \subset C$.
\end{lemma}
\begin{proof}
Recall that every open subset of $\mathbb{R}$ is the union of open intervals.
First we consider the case $\widehat{\mathcal{S}} = \mathcal{S}$.
By the property of the subspace topology on $\mathcal{S}$ relative to $V(\mathcal{S})$ (see Proposition \ref{prop:minimal_framework}), there are a finite number (say, $k$) of affine maps $f_i:\mathcal{S} \to [0,1]$ and intervals $I_i \subset [0,1]$, $1 \leq i \leq k$, such that $I_i$ are relatively open in $[0,1]$ and $s \in \bigcap_{i = 1}^{k} f_i^{-1}(I_i) \subset U$.
Now each $I_i$ is non-empty and the two endpoints of $I_i$ are different (as $I_i$ is relatively open in $[0,1]$), therefore $I_i$ contains an interval $J_i$ such that $J_i$ is also relatively open in $[0,1]$, $f_i(s) \in J_i$ and the closure $\overline{J_i}$ of $J_i$ relative to $[0,1]$ is also contained in $I_i$.
This implies that $s \in \bigcap_{i = 1}^{k} f_i^{-1}(J_i) \subset \bigcap_{i = 1}^{k} f_i^{-1}(\overline{J_i}) \subset \bigcap_{i = 1}^{k} f_i^{-1}(I_i) \subset U$, while $f_i^{-1}(J_i)$ is open and convex and $f_i^{-1}(\overline{J_i})$ is closed and convex (as $f_i$ is continuous and affine, and both $J_i$ and $\overline{J_i}$ are convex), therefore $C = \bigcap_{i = 1}^{k} f_i^{-1}(\overline{J_i})$ and $O = \bigcap_{i = 1}^{k} f_i^{-1}(J_i)$ satisfy the desired conditions.

The proof for the case $\widehat{\mathcal{S}} = \widetilde{\mathcal{S}}$ is similar.
By the property of the topology on $\widetilde{V}(\mathcal{S})$, there are a finite number (say, $k$) of continuous linear maps $f_i:\widetilde{V}(\mathcal{S}) \to \mathbb{R}$ and open intervals $I_i \subset \mathbb{R}$, $1 \leq i \leq k$, such that $s \in \widetilde{\mathcal{S}} \cap \bigcap_{i = 1}^{k} f_i^{-1}(I_i) \subset U$.
Now each $I_i$ contains an open interval $J_i$ such that $f_i(s) \in J_i$ and $\overline{J_i} \subset I_i$.
This implies that $s \in \widetilde{\mathcal{S}} \cap \bigcap_{i = 1}^{k} f_i^{-1}(J_i) \subset \widetilde{\mathcal{S}} \cap \bigcap_{i = 1}^{k} f_i^{-1}(\overline{J_i}) \subset U$, while $f_i^{-1}(J_i)$ is open and convex and $f_i^{-1}(\overline{J_i})$ is closed and convex by similar reasons.
Therefore $C = \widetilde{\mathcal{S}} \cap \bigcap_{i = 1}^{k} f_i^{-1}(\overline{J_i})$ and $O = \widetilde{\mathcal{S}} \cap \bigcap_{i = 1}^{k} f_i^{-1}(J_i)$ satisfy that $C$ is convex and closed, $O$ is convex and open, and $s \in O \subset C \subset U$.
The remaining claim now follows from Lemma \ref{lem:replace_with_essential}.
\end{proof}

\section{On the sets of affine maps on convex sets in general setting}
\label{sec:dynamics_topology}

In this section, we study the structural properties of the set $\mathcal{A}(\mathcal{S}_1,\mathcal{S}_2)$ of affine maps from $\mathcal{S}_1$ to $\mathcal{S}_2$, especially its topological properties.

\subsection{General cases}
\label{subsec:dynamics_topology_general}

First, we notice that the set $\mathcal{A}(\mathcal{S}_1,\mathcal{S}_2)$ also admits a natural convex structure.
Namely, for each $f,g \in \mathcal{A}(\mathcal{S}_1,\mathcal{S}_2)$ and $\lambda \in [0,1]$, we have $(\lambda f + (1-\lambda) g)(s) = \lambda f(s) + (1-\lambda) g(s) \in \mathcal{S}_2$ for every $s \in \mathcal{S}_1$, therefore $\lambda f + (1-\lambda) g \in \mathcal{A}(\mathcal{S}_1,\mathcal{S}_2)$.
This convex structure is compatible with composition of maps, namely we have
\begin{equation}
\begin{split}
h \circ (\lambda f + (1-\lambda) g) &= \lambda (h \circ f) + (1-\lambda) (h \circ g) \enspace, \\
(\lambda f + (1-\lambda) g) \circ h' &= \lambda (f \circ h') + (1-\lambda) (g \circ h')
\end{split}
\end{equation}
for each $f,g \in \mathcal{A}(\mathcal{S}_1,\mathcal{S}_2)$, $\lambda \in [0,1]$, $h \in \mathcal{A}(\mathcal{S}_2,\mathcal{S}_3)$, and $h' \in \mathcal{A}(\mathcal{S}_0,\mathcal{S}_1)$.
Moreover, each map in $\mathcal{A}(\mathcal{S}_1,\mathcal{S}_2)$ is continuous, namely we have:
\begin{lemma}
\label{lem:affine_is_continuous}
We have $\mathcal{A}(\mathcal{S}_1,\mathcal{S}_2) \subset C(\mathcal{S}_1,\mathcal{S}_2)$.
\end{lemma}
\begin{proof}
Recall from Proposition \ref{prop:minimal_framework} that the topology of $\mathcal{S}_i$ ($i = 1,2$) is the weak topology induced by the set $\mathcal{A}(\mathcal{S}_i,[0,1])$ of mappings.
Let $f \in \mathcal{A}(\mathcal{S}_1,\mathcal{S}_2)$.
Then for each $e \in \mathcal{A}(\mathcal{S}_2,[0,1])$, we have $e \circ f \in \mathcal{A}(\mathcal{S}_1,[0,1])$, therefore $e \circ f$ is continuous.
By the definition of weak topology, this implies that $f$ is continuous.
Hence the claim holds.
\end{proof}
By this lemma, when $\mathcal{S}_1$ and $\mathcal{S}_2$ are compact (hence $\widetilde{\mathcal{S}}_1 = \mathcal{S}_1$ and $\widetilde{\mathcal{S}}_2 = \mathcal{S}_2$), the set $\mathcal{A}^c(\widetilde{\mathcal{S}}_1,\widetilde{\mathcal{S}}_2)$ of all {\em continuous} affine maps $\widetilde{\mathcal{S}}_1 \to \widetilde{\mathcal{S}}_2$ coincides with $\mathcal{A}(\mathcal{S}_1,\mathcal{S}_2)$.

Let $\widetilde{\mathcal{A}}(\widetilde{\mathcal{S}}_1,\widetilde{\mathcal{S}}_2)$ denote the set of all $f \in \mathcal{A}^c(\widetilde{\mathcal{S}}_1,\widetilde{\mathcal{S}}_2)$ such that $f(\mathcal{S}_1) \subset \mathcal{S}_2$.
Then we have $\widetilde{\mathcal{A}}(\widetilde{\mathcal{S}}_1,\widetilde{\mathcal{S}}_2) \subset \mathcal{A}^c(\widetilde{\mathcal{S}}_1,\widetilde{\mathcal{S}}_2) \subset C(\widetilde{\mathcal{S}}_1,\widetilde{\mathcal{S}}_2)$ by the definition, while we have $\widetilde{\mathcal{A}}(\widetilde{\mathcal{S}}_1,\widetilde{\mathcal{S}}_2) = \mathcal{A}^c(\widetilde{\mathcal{S}}_1,\widetilde{\mathcal{S}}_2)$ if $\mathcal{S}_2$ is compact (i.e., $\widetilde{\mathcal{S}}_2 = \mathcal{S}_2$).
Now we have the following property:
\begin{lemma}
\label{lem:dynamics_virtual_states}
Each $f \in \mathcal{A}(\mathcal{S}_1,\mathcal{S}_2)$ extends to a unique $\widetilde{f} \in \widetilde{\mathcal{A}}(\widetilde{\mathcal{S}}_1,\widetilde{\mathcal{S}}_2)$, and every element of $\widetilde{\mathcal{A}}(\widetilde{\mathcal{S}}_1,\widetilde{\mathcal{S}}_2)$ is obtained from some $f \in \mathcal{A}(\mathcal{S}_1,\mathcal{S}_2)$ in this manner.
\end{lemma}
\begin{proof}
By Lemma \ref{lem:extends_to_affine} and Lemma \ref{lem:affine_is_continuous}, each $f \in \mathcal{A}(\mathcal{S}_1,\mathcal{S}_2)$ extends to a unique continuous affine map $\widetilde{f}:\widetilde{\mathcal{S}}_1 \to \widetilde{\mathcal{S}}_2$, and we have $\widetilde{f}(\mathcal{S}_1) = f(\mathcal{S}_1) \subset \mathcal{S}_2$.
The remaining claim is obvious.
\end{proof}
Hence the set $\mathcal{A}(\mathcal{S}_1,\mathcal{S}_2)$ can be identified via the map $f \mapsto \widetilde{f}$ with the subset $\widetilde{\mathcal{A}}(\widetilde{\mathcal{S}}_1,\widetilde{\mathcal{S}}_2)$ of $\mathcal{A}^c(\widetilde{\mathcal{S}}_1,\widetilde{\mathcal{S}}_2)$.
Note that $\widetilde{\mathcal{A}}(\widetilde{\mathcal{S}}_1,\widetilde{\mathcal{S}}_2)$ also admits a natural convex structure, as $\widetilde{f},\widetilde{g} \in \widetilde{\mathcal{A}}(\widetilde{\mathcal{S}}_1,\widetilde{\mathcal{S}}_2)$ and $\lambda \in [0,1]$ imply that $\lambda \widetilde{f} + (1-\lambda) \widetilde{g} \in \widetilde{\mathcal{A}}(\widetilde{\mathcal{S}}_1,\widetilde{\mathcal{S}}_2)$ (for the continuity of $\lambda \widetilde{f} + (1-\lambda) \widetilde{g}$, note that $\widetilde{f},\widetilde{g}$ and the operation on $\widetilde{\mathcal{S}}_1$ taking a convex combination of two elements are all continuous).
By virtue of the uniqueness property in Lemma \ref{lem:dynamics_virtual_states}, this convex structure and the correspondence $f \mapsto \widetilde{f}$ are compatible with convex structure on $\mathcal{A}(\mathcal{S}_1,\mathcal{S}_2)$, composition of mappings and some relevant objects.
Precisely, the following properties hold:
\begin{itemize}
\item If $h = \lambda f + (1-\lambda) g$, $f,g \in \mathcal{A}(\mathcal{S}_1,\mathcal{S}_2)$, then we have $\widetilde{h} = \lambda \widetilde{f} + (1-\lambda) \widetilde{g}$.
\item If $f \in \mathcal{A}(\mathcal{S}_1,\mathcal{S}_2)$, $g \in \mathcal{A}(\mathcal{S}_2,\mathcal{S}_3)$ and $h = g \circ f$, then $\widetilde{h} = \widetilde{g} \circ \widetilde{f}$.
\item We have $\widetilde{h} \circ (\lambda \widetilde{f} + (1-\lambda) \widetilde{g}) = \lambda (\widetilde{h} \circ \widetilde{f}) + (1-\lambda) (\widetilde{h} \circ \widetilde{g})$ for any $\widetilde{f},\widetilde{g} \in \widetilde{\mathcal{A}}(\widetilde{\mathcal{S}}_1,\widetilde{\mathcal{S}}_2)$, $\lambda \in [0,1]$ and $\widetilde{h} \in \widetilde{\mathcal{A}}(\widetilde{\mathcal{S}}_2,\widetilde{\mathcal{S}}_3)$.
\item We have $(\lambda \widetilde{f} + (1-\lambda) \widetilde{g}) \circ \widetilde{h} = \lambda (\widetilde{f} \circ \widetilde{h}) + (1-\lambda) (\widetilde{g} \circ \widetilde{h})$ for any $\widetilde{f},\widetilde{g} \in \widetilde{\mathcal{A}}(\widetilde{\mathcal{S}}_1,\widetilde{\mathcal{S}}_2)$, $\lambda \in [0,1]$ and $\widetilde{h} \in \widetilde{\mathcal{A}}(\widetilde{\mathcal{S}}_0,\widetilde{\mathcal{S}}_1)$.
\item We have $\widetilde{\mathrm{id}_{\mathcal{S}}} = \mathrm{id}_{\widetilde{\mathcal{S}}}$, where $\mathrm{id}_X$ denotes the identity map on a set $X$.
\item If $g \in \mathcal{A}(\mathcal{S}_2,\mathcal{S}_1)$ is a right (resp., left) inverse of $f \in \mathcal{A}(\mathcal{S}_1,\mathcal{S}_2)$, namely $f \circ g = \mathrm{id}_{\mathcal{S}_2}$ (resp., $g \circ f = \mathrm{id}_{\mathcal{S}_1}$), then $\widetilde{g} \in \widetilde{\mathcal{A}}(\widetilde{\mathcal{S}}_2,\widetilde{\mathcal{S}}_1)$ is a right (resp., left) inverse of $\widetilde{f} \in \widetilde{\mathcal{A}}(\widetilde{\mathcal{S}}_1,\widetilde{\mathcal{S}}_2)$.
Hence if $f \in \mathcal{A}(\mathcal{S}_1,\mathcal{S}_2)$ is invertible, then $\widetilde{f} \in \widetilde{\mathcal{A}}(\widetilde{\mathcal{S}}_1,\widetilde{\mathcal{S}}_2)$ is also invertible and $\widetilde{f}^{-1} = \widetilde{f^{-1}}$.
\end{itemize}

From now, we define topologies on the sets $\mathcal{A}(\mathcal{S}_1,\mathcal{S}_2)$ and $\mathcal{A}^c(\widetilde{\mathcal{S}}_1,\widetilde{\mathcal{S}}_2)$ which are analogy of the standard compact-open topology, by using the notion of essential subsets in Definition \ref{defn:essential_subset}:
\begin{definition}
\label{defn:modified_compact_open_topology}
First, we define the topology on the set $C(\widetilde{\mathcal{S}}_1,\widetilde{\mathcal{S}}_2)$ to be the topology generated by the family $\widetilde{\mathcal{B}}(\mathcal{S}_1,\mathcal{S}_2)$, referred to as the {\em subbase} of the topology, of all subsets of the form
\begin{equation}
O_{K,U} = \{f \in C(\widetilde{\mathcal{S}}_1,\widetilde{\mathcal{S}}_2) \mid f(K) \subset U\}
\end{equation}
such that $K$ is an essential closed (hence compact) subset of $\widetilde{\mathcal{S}}_1$ and $U$ is an essential open subset of $\widetilde{\mathcal{S}}_2$ (see Definition \ref{defn:essential_subset} for the terminology).
Namely, the open subsets of $C(\widetilde{\mathcal{S}}_1,\widetilde{\mathcal{S}}_2)$ are the arbitrary unions of finite intersections of members $O_{K,U}$ of $\widetilde{B}(\mathcal{S}_1,\mathcal{S}_2)$.
Then we define the topologies on $\mathcal{A}^c(\widetilde{\mathcal{S}}_1,\widetilde{\mathcal{S}}_2)$ and $\widetilde{\mathcal{A}}(\widetilde{\mathcal{S}}_1,\widetilde{\mathcal{S}}_2)$ to be the subspace topologies relative to $C(\widetilde{\mathcal{S}}_1,\widetilde{\mathcal{S}}_2)$, and define the topology on $\mathcal{A}(\mathcal{S}_1,\mathcal{S}_2)$ to be the topology induced from $\widetilde{\mathcal{A}}(\widetilde{\mathcal{S}}_1,\widetilde{\mathcal{S}}_2)$ via the bijection $f \mapsto \widetilde{f}$ given by Lemma \ref{lem:dynamics_virtual_states}.
\end{definition}
For simplicity, when we are focusing on the set $\mathcal{A}^c(\widetilde{\mathcal{S}}_1,\widetilde{\mathcal{S}}_2)$ rather than $C(\widetilde{\mathcal{S}}_1,\widetilde{\mathcal{S}}_2)$, we abuse the notation to write $O_{K,U}$ instead of $O_{K,U} \cap \mathcal{A}^c(\widetilde{\mathcal{S}}_1,\widetilde{\mathcal{S}}_2)$ unless some ambiguity occurs.
It is trivial from the definition that this topology coincides with the compact-open topology when both $\mathcal{S}_1$ and $\mathcal{S}_2$ are compact (i.e., $\widetilde{\mathcal{S}}_1 = \mathcal{S}_1$ and $\widetilde{\mathcal{S}}_2 = \mathcal{S}_2$), as every open or closed subset is essential in this case.
On the other hand, the definition suggests that the above topology would be in general weaker than the compact-open topology, as not every open or closed subset is essential.
Nevertheless, the following two properties still hold:
\begin{proposition}
\label{prop:dynamics_Hausdorff}
The topological space $C(\widetilde{\mathcal{S}}_1,\widetilde{\mathcal{S}}_2)$ is Hausdorff, hence so are $\mathcal{A}^c(\widetilde{\mathcal{S}}_1,\widetilde{\mathcal{S}}_2)$ and $\mathcal{A}(\mathcal{S}_1,\mathcal{S}_2)$.
\end{proposition}
\begin{proof}
Let $f,g \in C(\widetilde{\mathcal{S}}_1,\widetilde{\mathcal{S}}_2)$, $f \neq g$.
Then we have $f|_{\mathcal{S}_1} \neq g|_{\mathcal{S}_1}$, as $\mathcal{S}_1$ is dense in $\widetilde{\mathcal{S}}_1$ and $\widetilde{\mathcal{S}}_2$ is Hausdorff.
Choose an $s \in \mathcal{S}_1$ such that $f(s) \neq g(s)$.
As $\widetilde{\mathcal{S}}_2$ is Hausdorff, we have $U_1 \cap U_2 = \emptyset$ for some open neighborhoods $U_1,U_2$ in $\widetilde{\mathcal{S}}_2$ of $f(s),g(s)$, respectively.
By Lemma \ref{lem:closed_convex_neighborhood}, $U_1$ and $U_2$ can be chosen to be essential, while the set $K = \{s\}$ is closed and essential.
This implies that $f \in O_{K,U_1} \in \widetilde{\mathcal{B}}(\mathcal{S}_1,\mathcal{S}_2)$, $g \in O_{K,U_2} \in \widetilde{\mathcal{B}}(\mathcal{S}_1,\mathcal{S}_2)$ and $O_{K,U_1} \cap O_{K,U_2} = \emptyset$.
Hence $C(\widetilde{\mathcal{S}}_1,\widetilde{\mathcal{S}}_2)$ is Hausdorff, as desired.
\end{proof}
\begin{proposition}
\label{prop:equal_to_compact_open}
Suppose that $\mathcal{S}_1$ is compact, hence $\widetilde{\mathcal{S}}_1 = \mathcal{S}_1$.
Then the topology on $\mathcal{A}(\mathcal{S}_1,\mathcal{S}_2)$ defined above coincides with the compact-open topology.
\end{proposition}
\begin{proof}
In this situation, we have $\widetilde{f} = f$ for every $f \in \mathcal{A}(\mathcal{S}_1,\mathcal{S}_2)$, therefore $\mathcal{A}(\mathcal{S}_1,\mathcal{S}_2) = \widetilde{\mathcal{A}}(\widetilde{\mathcal{S}}_1,\widetilde{\mathcal{S}}_2)$.
Now for each $K \subset \mathcal{S}_1 = \widetilde{\mathcal{S}}_1$ and each $U \subset \widetilde{\mathcal{S}}_2$, we have
\begin{equation}
\begin{split}
O_{K,U} \cap \mathcal{A}(\mathcal{S}_1,\mathcal{S}_2)
&= \{f \in \mathcal{A}(\mathcal{S}_1,\mathcal{S}_2) \mid f(K) \subset U\} \\
&= \{f \in \mathcal{A}(\mathcal{S}_1,\mathcal{S}_2) \mid f(K) \subset U \cap \mathcal{S}_2\} \enspace.
\end{split}
\end{equation}
As the set $U \cap \mathcal{S}_2$ runs over all open subsets of $\mathcal{S}_2$ when $U$ runs over all essential open subsets of $\widetilde{\mathcal{S}}_2$ (see Lemma \ref{lem:essential_subset}(\ref{item:essential_subset_3})), the above relation implies that the subbase of the compact-open topology on $\mathcal{A}(\mathcal{S}_1,\mathcal{S}_2)$ coincides with the subbase of the topology on $\mathcal{A}(\mathcal{S}_1,\mathcal{S}_2)$ defined above.
Hence the claim holds.
\end{proof}

For the inclusion relations of the specified sets of mappings, we have the following three properties:
\begin{lemma}
\label{lem:A_is_closed}
\begin{enumerate}
\item \label{item:A_is_closed_1}
The subset $\mathcal{A}^c(\widetilde{\mathcal{S}}_1,\widetilde{\mathcal{S}}_2)$ of $C(\widetilde{\mathcal{S}}_1,\widetilde{\mathcal{S}}_2)$ is closed in $C(\widetilde{\mathcal{S}}_1,\widetilde{\mathcal{S}}_2)$.
\item \label{item:A_is_closed_2}
The subset $\mathcal{A}(\mathcal{S}_1,\mathcal{S}_2)$ of $C(\mathcal{S}_1,\mathcal{S}_2)$ is closed with respect to the compact-open topology on $C(\mathcal{S}_1,\mathcal{S}_2)$.
\end{enumerate}
\end{lemma}
\begin{proof}
To deal with the two claims in parallel, put $(\widehat{\mathcal{S}}_1,\widehat{\mathcal{S}}_2) = (\widetilde{\mathcal{S}}_1,\widetilde{\mathcal{S}}_2)$ and $(\widehat{\mathcal{S}}_1,\widehat{\mathcal{S}}_2) = (\mathcal{S}_1,\mathcal{S}_2)$ in the case of claim \ref{item:A_is_closed_1} and claim \ref{item:A_is_closed_2}, respectively.
Let $f \in C(\widehat{\mathcal{S}}_1,\widehat{\mathcal{S}}_2)$ which is not affine.
It suffices to show that there is an open neighborhood $O$ of $f$ in $C(\widehat{\mathcal{S}}_1,\widehat{\mathcal{S}}_2)$ such that each $g \in O$ is not affine.
First, we show that $f|_{\mathcal{S}_1}$ is not affine.
Assume for contrary that $f|_{\mathcal{S}_1}$ is affine.
Then by Lemma \ref{lem:extends_to_affine}, $f|_{\mathcal{S}_1}$ extends to a continuous affine map $f':\widehat{\mathcal{S}}_1 \to \widehat{\mathcal{S}}_2$.
Now note that $f$ and $f'$ are continuous, $\mathcal{S}_1$ is dense in $\widehat{\mathcal{S}}_1$, and $\widehat{\mathcal{S}}_2$ is Hausdorff.
Then it follows that $f = f'$, a contradiction.
Hence $f|_{\mathcal{S}_1}$ is not affine, therefore there are three elements $s_1,s_2,s_3 \in \mathcal{S}_1$ and a $\lambda \in [0,1]$ such that $s_3 = \lambda s_1 + (1-\lambda) s_2$ but $f(s_3) \neq \lambda f(s_1) + (1-\lambda) f(s_2)$.

As $\widehat{\mathcal{S}}_2$ is Hausdorff, there are open neighborhoods $V_1$ of $f(s_3)$ and $V_2$ of $\lambda f(s_1) + (1-\lambda) f(s_2)$, respectively, which are disjoint.
Moreover, as the operation of taking a convex combination in $\widehat{\mathcal{S}}_2$ is continuous, there are open neighborhoods $W_1$ of $f(s_1)$ and $W_2$ of $f(s_2)$, respectively, such that $\lambda t_1 + (1-\lambda) t_2 \in V_2$ for every $t_1 \in W_1$ and $t_2 \in W_2$.
In the case of claim \ref{item:A_is_closed_1}, Lemma \ref{lem:closed_convex_neighborhood} implies that these open neighborhoods $V_1$, $W_1$ and $W_2$ can be chosen to be essential.

Now let $O$ be the intersection of the three members $O_{\{s_1\},W_1}$, $O_{\{s_2\},W_2}$ and $O_{\{s_3\},V_1}$ of the subbase of the topology on $C(\widehat{\mathcal{S}}_1,\widehat{\mathcal{S}}_2)$.
This $O$ is an open neighborhood of $f$ in $C(\widehat{\mathcal{S}}_1,\widehat{\mathcal{S}}_2)$.
If $g \in O$, then we have $g(s_1) \in W_1$, $g(s_2) \in W_2$ and $g(s_3) \in V_1$, therefore $\lambda g(s_1) + (1-\lambda) g(s_2) \in V_2$ by the choice of $W_1$ and $W_2$, and $g(s_3) \neq \lambda g(s_1) + (1-\lambda) g(s_2)$ by the choice of $V_1$ and $V_2$.
Hence each $g \in O$ is not affine, concluding the proof of Lemma \ref{lem:A_is_closed}.
\end{proof}
\begin{remark}
\label{rem:A_is_closed}
More strongly, if $\mathcal{S}_1$ is compact, then Lemma \ref{lem:A_is_closed} and Proposition \ref{prop:equal_to_compact_open} imply that $\mathcal{A}(\mathcal{S}_1,\mathcal{S}_2)$ is a closed {\em topological subspace} of $C(\mathcal{S}_1,\mathcal{S}_2)$ with respect to the compact-open topology.
\end{remark}
\begin{lemma}
\label{lem:surjective_set_closed}
\begin{enumerate}
\item \label{item:surjective_set_closed_1}
The subset of all surjective maps in $\mathcal{A}^c(\widetilde{\mathcal{S}}_1,\widetilde{\mathcal{S}}_2)$ is closed in $\mathcal{A}^c(\widetilde{\mathcal{S}}_1,\widetilde{\mathcal{S}}_2)$.
\item \label{item:surjective_set_closed_2}
If $\mathcal{S}_1$ is compact, then the subset of all surjective maps in $\mathcal{A}(\mathcal{S}_1,\mathcal{S}_2)$ is closed in $\mathcal{A}(\mathcal{S}_1,\mathcal{S}_2)$.
\end{enumerate}
\end{lemma}
\begin{proof}
We prove the two claims in parallel.
Let $f \in \mathcal{A}^c(\widetilde{\mathcal{S}}_1,\widetilde{\mathcal{S}}_2)$ (resp., $f \in \mathcal{A}(\mathcal{S}_1,\mathcal{S}_2)$) which is not surjective.
Then $f(\widetilde{\mathcal{S}}_1)$ (resp., $f(\mathcal{S}_1)$) is a proper and compact subset of $\widetilde{\mathcal{S}}_2$ (resp., $\mathcal{S}_2$), therefore its complement in $\widetilde{\mathcal{S}}_2$ (resp., $\mathcal{S}_2$) is open and non-empty.
By Lemma \ref{lem:closed_convex_neighborhood}, there are a closed subset $C$ and an open subset $W$ of $\widetilde{\mathcal{S}}_2$ (resp., $\mathcal{S}_2$), respectively, such that $\emptyset \neq W \subset C$ and $C \subset f(\widetilde{\mathcal{S}}_1)^c$ (resp., $C \subset f(\mathcal{S}_1)^c$), therefore we have $f(\widetilde{\mathcal{S}}_1) \subset C^c \subset W^c \subsetneq \widetilde{\mathcal{S}}_2$ (resp., $f(\mathcal{S}_1) \subset C^c \subset W^c \subsetneq \mathcal{S}_2$).
Now we choose a subset $U$ of $W^c$ as follows: In the case of the claim \ref{item:surjective_set_closed_2}, we put $U = C^c$; while in the case of the claim \ref{item:surjective_set_closed_1}, by virtue of Lemma \ref{lem:replace_with_essential} we choose an essential open subset $U$ of $\widetilde{\mathcal{S}}_2$ such that $C^c \subset U \subset W^c$.
Then we have $f(\widetilde{\mathcal{S}}_1) \subset U \subsetneq \widetilde{\mathcal{S}}_2$ (resp., $f(\mathcal{S}_1) \subset U \subsetneq \mathcal{S}_2$), therefore the member $O_{\widetilde{\mathcal{S}}_1,U}$ (resp., $O_{\mathcal{S}_1,U}$) of the subbase of the topology on $\mathcal{A}^c(\widetilde{\mathcal{S}}_1,\widetilde{\mathcal{S}}_2)$ (resp., $\mathcal{A}(\mathcal{S}_1,\mathcal{S}_2)$) is an open neighborhood of $f$ in $\mathcal{A}^c(\widetilde{\mathcal{S}}_1,\widetilde{\mathcal{S}}_2)$ (resp., $\mathcal{A}(\mathcal{S}_1,\mathcal{S}_2)$) (this follows from Proposition \ref{prop:equal_to_compact_open} for the claim \ref{item:surjective_set_closed_2}, and from the fact that $\widetilde{\mathcal{S}}_1 = \overline{\mathcal{S}_1}$ is essential for the claim \ref{item:surjective_set_closed_1}).
Moreover, each element $g$ of $O_{\widetilde{\mathcal{S}}_1,U}$ (resp., $O_{\mathcal{S}_1,U}$) is not surjective, as we have $g(\widetilde{\mathcal{S}}_1) \subset U \subsetneq \widetilde{\mathcal{S}}_2$ (resp., $g(\mathcal{S}_1) \subset U \subsetneq \mathcal{S}_2$).
Hence the claim holds.
\end{proof}
\begin{proposition}
\label{prop:Atilde_dense_in_A}
If $\mathrm{int}_{\widetilde{V}(\mathcal{S}_2)}(\mathcal{S}_2) \neq \emptyset$, then the subset $\mathcal{A}^c(\widetilde{\mathcal{S}}_1,\mathcal{S}_2)$ of all continuous affine maps $f:\widetilde{\mathcal{S}}_1 \to \mathcal{S}_2$ is dense in $\mathcal{A}^c(\widetilde{\mathcal{S}}_1,\widetilde{\mathcal{S}}_2)$.
Hence $\widetilde{\mathcal{A}}(\widetilde{\mathcal{S}}_1,\widetilde{\mathcal{S}}_2) \supset \mathcal{A}^c(\widetilde{\mathcal{S}}_1,\mathcal{S}_2)$ is dense in $\mathcal{A}^c(\widetilde{\mathcal{S}}_1,\widetilde{\mathcal{S}}_2)$.
\end{proposition}
\begin{proof}
Let $O_{K_i,U_i}$ ($1 \leq i \leq k$) be members of the subbase of the topology on $\mathcal{A}^c(\widetilde{\mathcal{S}}_1,\widetilde{\mathcal{S}}_2)$ such that $O = \bigcap_{i=1}^{k} O_{K_i,U_i} \neq \emptyset$.
We show that $O \cap \mathcal{A}^c(\widetilde{\mathcal{S}}_1,\mathcal{S}_2) \neq \emptyset$.
Choose $f \in O$ and $x \in \mathrm{int}_{\widetilde{V}(\mathcal{S}_2)}(\mathcal{S}_2)$.
By applying Lemma \ref{lem:perturb_closed_set} to each pair $f(K_i) \subset U_i$ (note that $f(K_i)$ is compact as well as $K_i$), there exists an $m \in (0,1)$ such that $\lambda x + (1-\lambda) y \in U_i$ for every $1 \leq i \leq k$, $y \in f(K_i)$, and $0 \leq \lambda \leq m$.
Now we define a map $g:\widetilde{\mathcal{S}}_1 \to \widetilde{\mathcal{S}}_2$ by $g(s) = m x + (1-m) f(s)$ ($s \in \widetilde{\mathcal{S}}_1$).
Then $g$ is affine and continuous as well as $f$.
For each $s \in \widetilde{\mathcal{S}}_1$, we have $f(s) \in \widetilde{\mathcal{S}}_2 = \mathrm{cl}_{\widetilde{V}(\mathcal{S}_2)}(\mathcal{S}_2)$ and $0 < m < 1$, therefore it follows from \cite[Section II.1, Theorem 1.1]{tvs} that $m x + (1-m) f(s) \in \mathcal{S}_2$.
Hence $g \in \mathcal{A}^c(\widetilde{\mathcal{S}}_1,\mathcal{S})$.
Moreover, for each $1 \leq i \leq k$ and $s \in K_i$, we have $f(s) \in f(K_i)$ and $g(s) = m x + (1-m) f(s) \in U_i$ by the choice of $m$.
This implies that $g \in O$, therefore $g \in O \cap \mathcal{A}^c(\widetilde{\mathcal{S}}_1,\mathcal{S}_2)$ as desired.
Hence the claim holds.
\end{proof}

The next lemma says that any convex set $\mathcal{S}$ and its closure $\widetilde{\mathcal{S}}$ can be identified with the sets $\mathcal{A}(\ast,\mathcal{S})$ and $\mathcal{A}^c(\ast,\widetilde{\mathcal{S}}) = \mathcal{A}(\ast,\widetilde{\mathcal{S}})$, respectively, where $\ast$ denotes the convex set with just one element, hence several properties of convex sets can be immediately derived from those of the sets $\mathcal{A}(\mathcal{S}_1,\mathcal{S}_2)$:
\begin{lemma}
\label{lem:state_identified_with_dynamics}
For each $s \in \mathcal{S}$ (resp., $s \in \widetilde{\mathcal{S}}$), let $\iota_s$ denote the map $\ast \to \mathcal{S}$ (resp., $\ast \to \widetilde{\mathcal{S}}$) given by $\iota_s(\ast) = s$.
Then the map $\varphi:\mathcal{S} \to \mathcal{A}(\ast,\mathcal{S})$ (resp., $\varphi:\widetilde{\mathcal{S}} \to \mathcal{A}^c(\ast,\widetilde{\mathcal{S}})$), $\varphi(s) = \iota_s$, is an affine homeomorphism.
\end{lemma}
\begin{proof}
First we consider the case of $\widetilde{\mathcal{S}}$.
Note that $\iota_s \in \mathcal{A}^c(\ast,\widetilde{\mathcal{S}})$ for each $s \in \widetilde{\mathcal{S}}$, therefore $\varphi$ is well-defined.
It is obvious that $\varphi$ is affine and bijective.
Moreover, for each member $O_{\ast,U}$ of $\widetilde{\mathcal{B}}(\ast,\mathcal{S})$, we have $\varphi^{-1}(O_{\ast,U}) = U$ which is open in $\widetilde{\mathcal{S}}$, therefore $\varphi$ is continuous.
As $\widetilde{\mathcal{S}}$ is compact and $\mathcal{A}^c(\ast,\widetilde{\mathcal{S}})$ is Hausdorff (by Proposition \ref{prop:dynamics_Hausdorff}), a famous theorem in general topology implies that $\varphi$ is a homeomorphism, as desired.

Secondly, we consider the case of $\mathcal{S}$.
We have $\widetilde{\iota_s} = \iota_s$ for every $s \in \mathcal{S}$, therefore $\mathcal{A}(\ast,\mathcal{S})$ is naturally regarded as a subset of $\mathcal{A}^c(\ast,\widetilde{\mathcal{S}})$.
Now the map $\mathcal{S} \to \mathcal{A}(\ast,\mathcal{S})$ under consideration is bijective and it is the restriction to $\mathcal{S}$ of the map $\widetilde{\mathcal{S}} \to \mathcal{A}^c(\ast,\widetilde{\mathcal{S}})$ specified in the statement.
Hence the former map is also affine and homeomorphic, concluding the proof of Lemma \ref{lem:state_identified_with_dynamics}.
\end{proof}

Let $\mathcal{A}^*(\mathcal{S}_1,\mathcal{S}_2)$ (resp., $\mathcal{A}^{c*}(\widetilde{\mathcal{S}}_1,\widetilde{\mathcal{S}}_2)$) denote the set of all $f \in \mathcal{A}(\mathcal{S}_1,\mathcal{S}_2)$ (resp., $f \in \mathcal{A}^c(\widetilde{\mathcal{S}}_1,\widetilde{\mathcal{S}}_2)$) which are bijective.
From now, we show that the topology defined above is compatible with several operations and objects relevant to the sets of affine maps:
\begin{lemma}
\label{lem:convex_structure_continuous}
The map $\varphi:[0,1] \times \mathcal{A}^c(\widetilde{\mathcal{S}}_1,\widetilde{\mathcal{S}}_2) \times \mathcal{A}^c(\widetilde{\mathcal{S}}_1,\widetilde{\mathcal{S}}_2) \to \mathcal{A}^c(\widetilde{\mathcal{S}}_1,\widetilde{\mathcal{S}}_2)$ defined by $\varphi(\lambda,f,g) = \lambda f + (1-\lambda) g$ is continuous.
Namely, the operation of taking convex combination of two maps is continuous with respect to the above topology.
\end{lemma}
\begin{proof}
It suffices to show that $\varphi^{-1}(O_{K,U})$ is open for every $O_{K,U} \in \widetilde{\mathcal{B}}(\mathcal{S}_1,\mathcal{S}_2)$.
Let $\lambda \in [0,1]$ and $f,g \in \mathcal{A}^c(\widetilde{\mathcal{S}}_1,\widetilde{\mathcal{S}}_2)$ such that $\varphi(\lambda,f,g) \in O_{K,U}$.
Then we construct an open neighborhood $O$ of $(\lambda,f,g)$ such that $O \subset \varphi^{-1}(O_{K,U})$, in the following manner.

For each $s \in K$, we have $\lambda f(s) + (1-\lambda) g(s) = \varphi(\lambda,f,g)(s) \in U$.
As the operation of taking a convex combination of two elements in $\widetilde{\mathcal{S}}_2$ is continuous, there are an open neighborhood $I_s \subset [0,1]$ of $\lambda$ and open neighborhoods $V^f_s,V^g_s \subset \widetilde{\mathcal{S}}_2$ of $f(s),g(s)$, respectively, such that $\mu t^f + (1-\mu) t^g \in U$ for every $\mu \in I_s$, $t^f \in V^f_s$ and $t^g \in V^g_s$.
By Lemma \ref{lem:closed_convex_neighborhood}, these $V^f_s$ and $V^g_s$ can be chosen to be essential.
By Lemma \ref{lem:closed_convex_neighborhood} again, there are closed subsets $N^f_s,N^g_s$ of $\widetilde{\mathcal{S}}_2$ and open subsets $W^f_s,W^g_s$ of $\widetilde{\mathcal{S}}_2$ such that $f(s) \in W^f_s \subset N^f_s \subset V^f_s$ and $g(s) \in W^g_s \subset N^g_s \subset V^g_s$.

Now the map $f \times g:\widetilde{\mathcal{S}}_1 \to \widetilde{\mathcal{S}}_2 \times \widetilde{\mathcal{S}}_2$ defined by $(f \times g)(x) = (f(x),g(x))$ is continuous as both $f$ and $g$ are continuous, therefore $(f \times g)(K)$ is compact in $\widetilde{\mathcal{S}}_2 \times \widetilde{\mathcal{S}}_2$ as $K$ is compact.
On the other hand, $\{W^f_s \times W^g_s\}_{s \in K}$ is an open covering of $(f \times g)(K)$ by the choice of $W^f_s$ and $W^g_s$.
This implies that there are finitely many elements $s_1,\dots,s_k \in K$ such that $(f \times g)(K) \subset \bigcup_{i=1}^{k} (W^f_{s_i} \times W^g_{s_i})$.
We write $N^f_i = N^f_{s_i}$, $N^g_i = N^g_{s_i}$, $V^f_i = V^f_{s_i}$, and $V^g_i = V^g_{s_i}$ for simplicity.
Now we have $s_i \in f^{-1}(W^f_i) \subset f^{-1}(N^f_i)$ and $s_i \in g^{-1}(W^g_i) \subset g^{-1}(N^g_i)$, therefore Lemma \ref{lem:replace_with_essential} implies that we have $f^{-1}(W^f_i) \subset M^f_i \subset f^{-1}(N^f_i)$ and $g^{-1}(W^g_i) \subset M^g_i \subset g^{-1}(N^g_i)$ for some essential closed subsets $M^f_i$ and $M^g_i$.
We define
\begin{equation}
O = \bigcap_{i = 1}^{k} \left( I_{s_i} \times O_{M^f_i,V^f_i} \times O_{M^g_i,V^g_i} \right) \enspace.
\end{equation}
Note that $O_{M^f_i,V^f_i}$ and $O_{M^g_i,V^g_i}$ are members of $\widetilde{\mathcal{B}}(\mathcal{S}_1,\mathcal{S}_2)$, therefore $O$ is an open neighborhood of $(\lambda,f,g)$ in $[0,1] \times \mathcal{A}^c(\widetilde{\mathcal{S}}_1,\widetilde{\mathcal{S}}_2) \times \mathcal{A}^c(\widetilde{\mathcal{S}}_1,\widetilde{\mathcal{S}}_2)$.

We show that $\varphi(O) \subset O_{K,U}$.
Let $(\mu,f',g') \in O$.
Then for each $s \in K$, there is an index $i$ such that $(f(s),g(s)) \in W^f_i \times W^g_i$, therefore $s \in f^{-1}(W^f_i) \subset M^f_i$ and $s \in g^{-1}(W^g_i) \subset M^g_i$.
Now we have $\mu \in I_{s_i}$, $f'(s) \in V^f_i$ and $g'(s) \in V^g_i$, therefore $\mu f'(s) + (1-\mu) g'(s) \in U$ by the property mentioned in the second last paragraph.
This implies that $\varphi(\mu,f',g')(K) \subset U$, therefore $\varphi(\mu,f',g') \in O_{K,U}$, as desired.
Hence the proof of Lemma \ref{lem:convex_structure_continuous} is concluded.
\end{proof}
\begin{lemma}
\label{lem:composition_continuous}
The map $\varphi:\mathcal{A}^c(\widetilde{\mathcal{S}}_1,\widetilde{\mathcal{S}}_2) \times \mathcal{A}^c(\widetilde{\mathcal{S}}_2,\widetilde{\mathcal{S}}_3) \to \mathcal{A}^c(\widetilde{\mathcal{S}}_1,\widetilde{\mathcal{S}}_3)$ defined by $\varphi(f,g) = g \circ f$ is continuous.
Namely, the operation of composing two maps is continuous.
\end{lemma}
\begin{proof}
It suffices to show that $\varphi^{-1}(O_{K,U})$ is open for every $O_{K,U} \in \widetilde{\mathcal{B}}(\mathcal{S}_1,\mathcal{S}_2)$.
Let $f \in \mathcal{A}^c(\widetilde{\mathcal{S}}_1,\widetilde{\mathcal{S}}_2)$ and $g \in \mathcal{A}^c(\widetilde{\mathcal{S}}_2,\widetilde{\mathcal{S}}_3)$ such that $g \circ f \in O_{K,U}$.
Then we have $f(K) \subset g^{-1}(U) \subset \widetilde{\mathcal{S}}_2$, while $f(K)$ is compact (as $K$ is compact) and $g^{-1}(U)$ is open.
Now for each $s \in K$, Lemma \ref{lem:closed_convex_neighborhood} (applied twice) implies that there are open subsets $V_s,W_s$ of $\widetilde{\mathcal{S}}_2$ and closed subsets $M_s,N_s$ of $\widetilde{\mathcal{S}}_2$ such that $V_s$ is essential and $f(s) \in W_s \subset N_s \subset V_s \subset M_s \subset g^{-1}(U)$.
As $f(K)$ is compact, there are finitely many elements $s_1,\dots,s_k \in K$ such that $f(K) \subset \bigcup_{i = 1}^{k} W_{s_i}$.
Write $W_i = W_{s_i}$, $N_i = N_{s_i}$, $V_i = V_{s_i}$, and $M_i = M_{s_i}$.
Now we have $s_i \in f^{-1}(W_i) \subset f^{-1}(N_i)$, therefore Lemma \ref{lem:replace_with_essential} implies that there is an essential closed subset $N'_i$ such that $f^{-1}(W_i) \subset N'_i \subset f^{-1}(N_i)$.
Moreover, we have $\bigcup_{i=1}^{k} V_i \subset \bigcup_{i=1}^{k} M_i$, therefore Lemma \ref{lem:replace_with_essential} implies that there is an essential closed subset $M'$ such that $\bigcup_{i=1}^{k} V_i \subset M' \subset \bigcup_{i=1}^{k} M_i$.
Now put $O_f = \bigcap_{i=1}^{k} O_{N'_i,V_i}$ and $O_g = O_{M',U}$.
Note that $O_f$ and $O_g$ are open neighborhoods of $f$ and $g$, respectively.
We show that $\varphi(O_f \times O_g) \subset O_{K,U}$.
Let $f' \in O_f$ and $g' \in O_g$.
Then for each $s \in K$, there is an index $i$ such that $s \in f^{-1}(W_i)$.
Now we have $s \in N'_i$, therefore $f'(s) \in V_i \subset M'$ and $g'(f'(s)) \in U$.
This implies that $g'(f'(K)) \subset U$, therefore we have $\varphi(f',g') = g' \circ f' \in O_{K,U}$, as desired.
Hence the claim holds.
\end{proof}
\begin{corollary}
\label{cor:evaluation_map_continuous}
The evaluation map $\mathcal{A}^c(\widetilde{\mathcal{S}}_1,\widetilde{\mathcal{S}}_2) \times \widetilde{\mathcal{S}}_1 \to \widetilde{\mathcal{S}}_2$, $(f,s) \mapsto f(s)$, is continuous.
\end{corollary}
\begin{proof}
This follows from Lemma \ref{lem:composition_continuous} and Lemma \ref{lem:state_identified_with_dynamics}.
\end{proof}
\begin{lemma}
\label{lem:inverse_homeomorphic}
The map $\varphi:\mathcal{A}^{c*}(\widetilde{\mathcal{S}}_1,\widetilde{\mathcal{S}}_2) \to \mathcal{A}^{c*}(\widetilde{\mathcal{S}}_2,\widetilde{\mathcal{S}}_1)$, $\varphi(f) = f^{-1}$, is a homeomorphism.
Namely, the operation of taking the inverse of a map is homeomorphic.
\end{lemma}
\begin{proof}
First, for each $f \in \mathcal{A}^{c*}(\widetilde{\mathcal{S}}_1,\widetilde{\mathcal{S}}_2)$, $f^{-1}$ is also an affine map, while $f$ is a homeomorphism as $f$ is a continuous bijection from the compact $\widetilde{\mathcal{S}}_1$ to the Hausdorff $\widetilde{\mathcal{S}}_2$.
Hence we have $f^{-1} \in \mathcal{A}^{c*}(\widetilde{\mathcal{S}}_2,\widetilde{\mathcal{S}}_1)$ and $\varphi$ is well-defined.
Note that $\varphi$ is obviously a bijection.
To prove that $\varphi$ is continuous, we show that $\varphi^{-1}(O_{K,U} \cap \mathcal{A}^{c*}(\widetilde{\mathcal{S}}_2,\widetilde{\mathcal{S}}_1)) = O_{U^c,K^c} \cap \mathcal{A}^{c*}(\widetilde{\mathcal{S}}_1,\widetilde{\mathcal{S}}_2)$ for every $O_{K,U} \in \widetilde{\mathcal{B}}(\mathcal{S}_2,\mathcal{S}_1)$ (note that $O_{U^c,K^c} \in \widetilde{\mathcal{B}}(\mathcal{S}_1,\mathcal{S}_2)$, as $U^c$ is closed and essential in $\widetilde{\mathcal{S}}_1$ and $K^c$ is open and essential in $\widetilde{\mathcal{S}}_2$ by Lemma \ref{lem:essential_subset}).
Indeed, for every $g \in O_{K,U} \cap \mathcal{A}^{c*}(\widetilde{\mathcal{S}}_2,\widetilde{\mathcal{S}}_1)$, we have $g(K) \subset U$, therefore $g^{-1}(U^c) \subset K^c$ and $g^{-1} \in O_{U^c,K^c}$ as $g$ is a bijection.
This implies the inclusion $\subset$, and the other inclusion $\supset$ holds similarly.
Therefore $\varphi$ is continuous, and the same argument shows that $\varphi^{-1}$ is also continuous.
Hence $\varphi$ is a homeomorphism, as desired.
\end{proof}
\begin{corollary}
\label{cor:invertible_dynamics_topological_group}
The set $\mathcal{A}^*(\mathcal{S},\mathcal{S})$ (resp., $\mathcal{A}^{c*}(\widetilde{\mathcal{S}},\widetilde{\mathcal{S}})$) forms a topological group with map composition as multiplication, and this group acts continuously on $\mathcal{S}$ (resp., $\widetilde{\mathcal{S}}$) by $f \cdot s = f(s)$ for each $f \in \mathcal{A}^*(\mathcal{S},\mathcal{S})$ (resp., $f \in \mathcal{A}^{c*}(\widetilde{\mathcal{S}},\widetilde{\mathcal{S}})$) and $s \in \mathcal{S}$ (resp., $s \in \widetilde{\mathcal{S}})$.
\end{corollary}
\begin{proof}
This follows from Lemma \ref{lem:composition_continuous}, Lemma \ref{lem:inverse_homeomorphic}, and Corollary \ref{cor:evaluation_map_continuous}.
\end{proof}

Finally, we discuss the following sort of \lq\lq coordinate expressions'' of affine maps.
We choose and fix any (possibly infinite) subset $B \subset \mathcal{S}_1$, and define an affine map $\iota_B$ from $\mathcal{A}(\mathcal{S}_1,\mathcal{S}_2)$ (resp., $\mathcal{A}^c(\widetilde{\mathcal{S}}_1,\widetilde{\mathcal{S}}_2)$) to the set $\mathcal{S}_2{}^B$ (resp., $\widetilde{\mathcal{S}}_2{}^B$) of all mappings $B \to \mathcal{S}_2$ (resp., $B \to \widetilde{\mathcal{S}}_2$) by $\iota_B(f) = f|_B$, $f \in \mathcal{A}(\mathcal{S}_1,\mathcal{S}_2)$ (resp., $f \in \mathcal{A}^c(\widetilde{\mathcal{S}}_1,\widetilde{\mathcal{S}}_2)$).
Note that $\iota_B$ is injective if $\mathcal{S}_1 \subset \mathrm{Aff}(B)$ (see Lemma \ref{lem:extends_to_affine}).
Now the set $\mathcal{S}_2{}^B$ (resp., $\widetilde{\mathcal{S}}_2{}^B$) admits a natural topology that is the weakest topology to make, for every $b \in B$, the \lq\lq evaluation map'' $\mathrm{ev}_b:\mathcal{S}_2{}^B \ni f \mapsto f(b) \in \mathcal{S}_2$ (resp., $\mathrm{ev}_b:\widetilde{\mathcal{S}}_2{}^B \ni f \mapsto f(b) \in \widetilde{\mathcal{S}}_2$) continuous.
This topological space is nothing but the direct product space of copies of $\mathcal{S}_2$ (resp., $\widetilde{\mathcal{S}}_2$) over $b \in B$, therefore $\mathcal{S}_2{}^B$ (resp., $\widetilde{\mathcal{S}}_2{}^B$) is Hausdorff as well as $\mathcal{S}_2$ (resp., $\widetilde{\mathcal{S}}_2$).
By the Tychonoff's Theorem, $\widetilde{\mathcal{S}}_2{}^B$ is compact, while $\mathcal{S}_2{}^B$ is compact if (and only if) $\mathcal{S}_2$ is compact.
Now we have the following property:
\begin{lemma}
\label{lem:iota_B_continuous}
Under the above setting, the map $\iota_B$ is continuous.
\end{lemma}
\begin{proof}
First, we prove that $\iota_B:\mathcal{A}^c(\widetilde{\mathcal{S}}_1,\widetilde{\mathcal{S}}_2) \to \widetilde{\mathcal{S}}_2{}^B$ is continuous.
It suffices to show that for each $b \in B$, the composite map $\varphi_b = \mathrm{ev}_b \circ \iota_B:\mathcal{A}^c(\widetilde{\mathcal{S}}_1,\widetilde{\mathcal{S}}_2) \to \widetilde{\mathcal{S}}_2$ given by $\varphi_b(f) = f(b)$ is continuous.
Now by Lemma \ref{lem:closed_convex_neighborhood}, each open subset of $\widetilde{\mathcal{S}}_2$ is the union of essential open subsets $U$ of $\widetilde{\mathcal{S}}_2$, and we have $\varphi_b^{-1}(U) = O_{\{b\},U} \in \widetilde{\mathcal{B}}(\mathcal{S}_1,\mathcal{S}_2)$ for any such $U$ (note that $\{b\} \subset B \subset \mathcal{S}_1$, therefore the closed subset $\{b\}$ is essential).
This implies that $\varphi_b$ is continuous, as desired.

Secondly, for the map $\iota_B:\mathcal{A}(\mathcal{S}_1,\mathcal{S}_2) \to \mathcal{S}_2{}^B$, this map is the composition of the continuous map $\mathcal{A}(\mathcal{S}_1,\mathcal{S}_2) \to \mathcal{A}^c(\widetilde{\mathcal{S}}_1,\widetilde{\mathcal{S}}_2)$, $f \mapsto \widetilde{f}$ given by Lemma \ref{lem:dynamics_virtual_states}, followed by the continuous map $\mathcal{A}^c(\widetilde{\mathcal{S}}_1,\widetilde{\mathcal{S}}_2) \to \widetilde{\mathcal{S}}_2{}^B$, $f \mapsto f|_B$ studied in the previous paragraph.
Hence $\iota_B$ is also continuous.
\end{proof}

\subsection{Finite-dimensional cases}
\label{subsec:finite_dimensional}

In this subsection, we study the case of finite-dimensional convex sets and give some enhancements of the above general results.
First note that, if $\dim(\mathcal{S}) = n < \infty$, then the locally convex Hausdorff topological vector space $V(\mathcal{S})$ is naturally identified with the $n$-dimensional Euclidean space $\mathbb{R}^n$, we have $\widetilde{V}(\mathcal{S}) = V(\mathcal{S})$, and $\widetilde{\mathcal{S}}$ is the closure of $\mathcal{S}$ in $V(\mathcal{S}) = \mathbb{R}^n$.
Now we have the following three properties for inclusion relations of the sets of affine maps:
\begin{lemma}
\label{lem:invertible_set_closed}
Suppose that $\dim(\mathcal{S}_1) < \infty$.
\begin{enumerate}
\item \label{item:invertible_set_closed_1}
If $\mathcal{A}^{c*}(\widetilde{\mathcal{S}}_1,\widetilde{\mathcal{S}}_2) \neq \emptyset$, then we have $\dim(\mathcal{S}_1) = \dim(\mathcal{S}_2)$, every surjective map in $\mathcal{A}^c(\widetilde{\mathcal{S}}_1,\widetilde{\mathcal{S}}_2)$ is a bijection, and $\mathcal{A}^{c*}(\widetilde{\mathcal{S}}_1,\widetilde{\mathcal{S}}_2)$ is closed in $\mathcal{A}^c(\widetilde{\mathcal{S}}_1,\widetilde{\mathcal{S}}_2)$.
\item \label{item:invertible_set_closed_2}
If $\mathcal{A}^*(\mathcal{S}_1,\mathcal{S}_2) \neq \emptyset$, then we have $\dim(\mathcal{S}_1) = \dim(\mathcal{S}_2)$ and every surjective map in $\mathcal{A}(\mathcal{S}_1,\mathcal{S}_2)$ is a bijection.
Moreover, if $\mathcal{S}_1$ is compact, then $\mathcal{A}^*(\mathcal{S}_1,\mathcal{S}_2)$ is closed in $\mathcal{A}(\mathcal{S}_1,\mathcal{S}_2)$.
\end{enumerate}
\end{lemma}
\begin{proof}
We prove the two claims in parallel.
First note that $\widetilde{V}(\mathcal{S}_1) = V(\mathcal{S}_1)$ by the assumption.
Take an $f \in \mathcal{A}^{c*}(\widetilde{\mathcal{S}}_1,\widetilde{\mathcal{S}}_2)$ (resp., $f \in \mathcal{A}^*(\mathcal{S}_1,\mathcal{S}_2)$).
Then, as the affine hulls of $\widetilde{\mathcal{S}}_1$ and $\widetilde{\mathcal{S}}_2$ (resp., $\mathcal{S}_1$ and $\mathcal{S}_2$) are the whole underlying spaces, the maps $f$ and $g = f^{-1}$ extend uniquely to affine maps $\overline{f}:V(\mathcal{S}_1) \to \widetilde{V}(\mathcal{S}_2)$ and $\overline{g}:\widetilde{V}(\mathcal{S}_2) \to V(\mathcal{S}_1)$ (resp., $\overline{f}:V(\mathcal{S}_1) \to V(\mathcal{S}_2)$ and $\overline{g}:V(\mathcal{S}_2) \to V(\mathcal{S}_1)$), respectively.
Now we have $\overline{g} \circ \overline{f} = \mathrm{id}_{V(\mathcal{S}_1)}$, as both of the two maps in the left-hand and the right-hand sides are affine extensions of $g \circ f = \mathrm{id}_{\widetilde{\mathcal{S}}_1}$ (resp., $g \circ f = \mathrm{id}_{\mathcal{S}_1}$).
Similarly, we also have $\overline{f} \circ \overline{g} = \mathrm{id}_{\widetilde{V}(\mathcal{S}_2)}$ (resp., $\overline{f} \circ \overline{g} = \mathrm{id}_{V(\mathcal{S}_2)}$).
Therefore $V(\mathcal{S}_1)$ and $\widetilde{V}(\mathcal{S}_2)$ (resp., $V(\mathcal{S}_2)$) are affine isomorphic, hence $\dim(\widetilde{\mathcal{S}}_2) = \dim(\mathcal{S}_1) < \infty$ (resp., $\dim(\mathcal{S}_2) = \dim(\mathcal{S}_1) < \infty$).
This implies that $\dim(\widetilde{\mathcal{S}}_2) = \dim(\mathcal{S}_2)$ and $\widetilde{V}(\mathcal{S}_2) = V(\mathcal{S}_2)$, therefore $\dim(\widetilde{\mathcal{S}}_2) = \dim(\mathcal{S}_2) = \dim(\mathcal{S}_1)$.
Now each surjective $h \in \mathcal{A}^c(\widetilde{\mathcal{S}}_1,\widetilde{\mathcal{S}}_2)$ (resp., $h \in \mathcal{A}(\mathcal{S}_1,\mathcal{S}_2)$) extends to an affine map $\overline{h}:V(\mathcal{S}_1) \to V(\mathcal{S}_2)$, which is also surjective as the image of $\overline{h}$ is convex and contains $\mathcal{S}_2$.
As $V(\mathcal{S}_1)$ and $V(\mathcal{S}_2)$ have the same finite dimension, this surjective affine map $\overline{h}$ is also injective, so is $h$ (hence $h$ is bijective).
Finally, the remaining part of the claim now follows from Lemma \ref{lem:surjective_set_closed}.
\end{proof}
\begin{corollary}
\label{cor:invertible_dynamics_closed_topological_group}
Suppose that $\dim(\mathcal{S}) < \infty$.
Then the topological group $\mathcal{A}^{c*}(\widetilde{\mathcal{S}},\widetilde{\mathcal{S}})$ is closed in $\mathcal{A}^c(\widetilde{\mathcal{S}},\widetilde{\mathcal{S}})$.
Moreover, if $\mathcal{S}$ is compact, then the topological group $\mathcal{A}^*(\mathcal{S},\mathcal{S})$ is closed in $\mathcal{A}(\mathcal{S},\mathcal{S})$.
\end{corollary}
\begin{proof}
This follows immediately from Lemma \ref{lem:invertible_set_closed}.
\end{proof}
\begin{proposition}
\label{prop:Atilde_dense_in_A_finite_dim}
If $\dim(\mathcal{S}_2) < \infty$, then both $\mathcal{A}^c(\widetilde{\mathcal{S}}_1,\mathcal{S}_2)$ and $\widetilde{\mathcal{A}}(\widetilde{\mathcal{S}}_1,\widetilde{\mathcal{S}}_2)$ are dense in $\mathcal{A}^c(\widetilde{\mathcal{S}}_1,\widetilde{\mathcal{S}}_2)$.
\end{proposition}
\begin{proof}
Under the assumption, $\widetilde{V}(\mathcal{S}_2) = V(\mathcal{S}_2)$ is a finite-dimensional Euclidean space such that $\mathrm{Aff}(\mathcal{S}_2) = V(\mathcal{S}_2)$, therefore $\mathrm{int}_{\widetilde{V}(\mathcal{S}_2)}(\mathcal{S}_2) \neq \emptyset$ (see e.g., \cite[Proposition 2.1.7]{Gru}).
Hence the condition in Proposition \ref{prop:Atilde_dense_in_A} is satisfied, therefore the claim follows from Proposition \ref{prop:Atilde_dense_in_A}.
\end{proof}

When $\mathcal{S}$ is finite-dimensional, let $d_{\mathcal{S}}$ denote the Euclidean metric on $\widetilde{V}(\mathcal{S}) = V(\mathcal{S}) = \mathbb{R}^{\dim(\mathcal{S})}$.
Then $\widetilde{\mathcal{S}}$ is a compact (hence complete) metric space with respect to $d_{\mathcal{S}}$.
Now if both $\mathcal{S}_1$ and $\mathcal{S}_2$ are finite-dimensional, then the set $C(\widetilde{\mathcal{S}}_1,\widetilde{\mathcal{S}}_2)$ is also a complete metric space with respect to the metric $d_{\infty}$ defined by
\begin{equation}
d_{\infty}(f,g) = \sup_{s \in \widetilde{\mathcal{S}}_1} d_{\mathcal{S}_2}(f(s),g(s)) \mbox{ for } f,g \in C(\widetilde{\mathcal{S}}_1,\widetilde{\mathcal{S}}_2) \enspace.
\end{equation}
(In fact, for each pair $(f,g)$ the supremum is attained by some $s \in \widetilde{\mathcal{S}}_1$, as the function $s \mapsto d_{\mathcal{S}_2}(f(s),g(s))$ is continuous on the compact space $\widetilde{\mathcal{S}}_1$.)
For any sequence in $C(\widetilde{\mathcal{S}}_1,\widetilde{\mathcal{S}}_2)$, convergence with respect to the metric $d_{\infty}$ is equivalent to the uniform convergence of mappings over $\widetilde{\mathcal{S}}_1$.
Note that the topology on $C(\widetilde{\mathcal{S}}_1,\widetilde{\mathcal{S}}_2)$ determined by $d_{\infty}$ coincides with the compact-open topology (see e.g., \cite[Section c-20.5]{top_book}).
To avoid confusion, we write $C_{d_{\infty}}(\widetilde{\mathcal{S}}_1,\widetilde{\mathcal{S}}_2)$ to signify the set $C(\widetilde{\mathcal{S}}_1,\widetilde{\mathcal{S}}_2)$ endowed with the topology determined by $d_{\infty}$ rather than the original topology defined in Definition \ref{defn:modified_compact_open_topology}.
Note that the topology of $C(\widetilde{\mathcal{S}}_1,\widetilde{\mathcal{S}}_2)$ is weaker than or equal to that of $C_{d_{\infty}}(\widetilde{\mathcal{S}}_1,\widetilde{\mathcal{S}}_2)$, while these are equal if both $\mathcal{S}_1$ and $\mathcal{S}_2$ are compact.
Now we have the following fundamental result:
\begin{proposition}
\label{prop:A_is_compact}
If both $\mathcal{S}_1$ and $\mathcal{S}_2$ are finite-dimensional, then the topology on $\mathcal{A}^c(\widetilde{\mathcal{S}}_1,\widetilde{\mathcal{S}}_2)$ coincides with the topology determined by the metric $d_{\infty}$ (or equivalently, the compact-open topology), and $\mathcal{A}^c(\widetilde{\mathcal{S}}_1,\widetilde{\mathcal{S}}_2)$ is compact.
\end{proposition}
The key fact for the proof is the Arzel\`{a}-Ascoli Theorem (see e.g., \cite[Section c-20.6]{top_book}).
The statement of the theorem relevant to our present situation is the following (note that the other condition for \lq\lq boundedness'' is now automatically satisfied, as the metric space $\widetilde{\mathcal{S}}_2$ is compact):
\begin{theorem}
[Arzel\`{a}-Ascoli]
\label{thm:Ascoli-Arzela}
A subset $\mathcal{F}$ of $C_{d_{\infty}}(\widetilde{\mathcal{S}}_1,\widetilde{\mathcal{S}}_2)$ has compact closure if and only if $\mathcal{F}$ is equicontinuous, i.e., for any $s \in \widetilde{\mathcal{S}}_1$ and any $\varepsilon > 0$, there exists an open neighborhood $U$ of $s$ such that for every $f \in \mathcal{F}$, we have $d_{\mathcal{S}_2}(f(s),f(t)) < \varepsilon$ whenever $t \in U$.
\end{theorem}
\begin{proof}
[Proof of Proposition \ref{prop:A_is_compact}]
First, we prove that the set $\mathcal{F} = \mathcal{A}^c(\widetilde{\mathcal{S}}_1,\widetilde{\mathcal{S}}_2)$ is compact with respect to the metric $d_{\infty}$.
Note that the subset $\mathcal{F}$ is closed in $C_{d_{\infty}}(\widetilde{\mathcal{S}}_1,\widetilde{\mathcal{S}}_2)$, as $\mathcal{F}$ is closed in $C(\widetilde{\mathcal{S}}_1,\widetilde{\mathcal{S}}_2)$ by Lemma \ref{lem:A_is_closed} and the topology on $C_{d_{\infty}}(\widetilde{\mathcal{S}}_1,\widetilde{\mathcal{S}}_2)$ is stronger than (or equal to) that on $C(\widetilde{\mathcal{S}}_1,\widetilde{\mathcal{S}}_2)$.
Therefore, by Theorem \ref{thm:Ascoli-Arzela}, it suffices to show that $\mathcal{F}$ is equicontinuous.
Let $s_0 \in \widetilde{\mathcal{S}}_1$ and $\varepsilon > 0$.
Choose $s_1,\dots,s_n \in \widetilde{\mathcal{S}}_1$ such that $\{s_0,\dots,s_n\}$ is affine independent and its affine hull is $V(\mathcal{S}_1)$.
Then for each $s \in \widetilde{\mathcal{S}}_1$, there exists a unique expression $s - s_0 = \sum_{i = 1}^{n} \lambda_i(s) (s_i - s_0)$ with $\lambda_i(s) \in \mathbb{R}$.
Note that every $\lambda_i$ is a continuous function on $\widetilde{\mathcal{S}}_1$, as any exchange of coordinate systems in a finite-dimensional Euclidean space is a continuous operation.
Now let $U$ be the set of all $s \in \widetilde{\mathcal{S}}_1$ such that $c \sum_i |\lambda_i(s)| < \varepsilon$, where $c = \sup_{t,t' \in \widetilde{\mathcal{S}}_2} d_{\mathcal{S}_2}(t,t') < \infty$ (the finiteness of $c$ follows from the compactness of $\widetilde{\mathcal{S}}_2$).
Then $U$ is open as every $\lambda_i$ is continuous, and $s_0 \in U$ as $\lambda_i(s_0) = 0$ for every $i$.
Moreover, for each $f \in \mathcal{F}$ and $s \in U$, we have
\begin{equation}
f(s) = f \left( \left(1 - \sum_i \lambda_i(s) \right) s_0 + \sum_i \lambda_i(s) s_i \right)
= \left(1 - \sum_i \lambda_i(s) \right) f(s_0) + \sum_i \lambda_i(s) f(s_i) \enspace,
\end{equation}
therefore $f(s) - f(s_0) = \sum_i \lambda_i(s) (f(s_i) - f(s_0))$.
This implies that
\begin{equation}
d_{\mathcal{S}_2}(f(s),f(s_0)) \leq \sum_i |\lambda_i(s)| d_{\mathcal{S}_2}(f(s_i),f(s_0)) \leq c \sum_i |\lambda_i(s)| < \varepsilon \enspace.
\end{equation}
Hence $\mathcal{F}$ is equicontinuous, therefore $\mathcal{F}$ is compact with respect to the metric $d_{\infty}$.

Let $\varphi$ denote the identity map on the set $C(\widetilde{\mathcal{S}}_1,\widetilde{\mathcal{S}}_2)$, which is regarded as the map between topological spaces $\varphi:C_{d_{\infty}}(\widetilde{\mathcal{S}}_1,\widetilde{\mathcal{S}}_2) \to C(\widetilde{\mathcal{S}}_1,\widetilde{\mathcal{S}}_2)$.
This map $\varphi$ is continuous, as the topology of $C_{d_{\infty}}(\widetilde{\mathcal{S}}_1,\widetilde{\mathcal{S}}_2)$ is stronger than (or equal to) the topology of $C(\widetilde{\mathcal{S}}_1,\widetilde{\mathcal{S}}_2)$.
Now $\mathcal{F} = \mathcal{A}^c(\widetilde{\mathcal{S}}_1,\widetilde{\mathcal{S}}_2)$ is Hausdorff with respect to the subspace topology relative to $C(\widetilde{\mathcal{S}}_1,\widetilde{\mathcal{S}}_2)$ (Proposition \ref{prop:dynamics_Hausdorff}), while $\mathcal{F}$ is compact with respect to the metric $d_{\infty}$ as above.
This implies that the restriction $\varphi|_{\mathcal{F}}:\mathcal{F} \to \mathcal{F}$ of $\varphi$ to $\mathcal{F}$ is a homeomorphism as the domain is compact and the range is Hausdorff, therefore the two topologies on $\mathcal{F}$ coincide with each other.
Hence Proposition \ref{prop:A_is_compact} holds.
\end{proof}

As an application of this result, we study the continuous map $\iota_B:\mathcal{A}^c(\widetilde{\mathcal{S}}_1,\widetilde{\mathcal{S}}_2) \to \widetilde{\mathcal{S}}_2{}^B$ discussed in Lemma \ref{lem:iota_B_continuous}.
Let $B$ be a subset of $\mathcal{S}_1$ such that $\mathcal{S}_1 \subset \mathrm{Aff}(B)$, hence $\iota_B$ is injective.
Suppose that both $\mathcal{S}_1$ and $\mathcal{S}_2$ are finite-dimensional.
Then we have the following:
\begin{corollary}
\label{cor:iota_B_homeomorphism}
Under the above setting, the continuous map $\iota_B:\mathcal{A}^c(\widetilde{\mathcal{S}}_1,\widetilde{\mathcal{S}}_2) \to \widetilde{\mathcal{S}}_2{}^B$ is a homeomorphism onto its image.
\end{corollary}
\begin{proof}
Now $\mathcal{A}^c(\widetilde{\mathcal{S}}_1,\widetilde{\mathcal{S}}_2)$ is compact by Proposition \ref{prop:A_is_compact}, while $\widetilde{\mathcal{S}}_2{}^B$ is Hausdorff.
Hence the continuous injection $\iota_B$ is a homeomorphism onto its image.
\end{proof}
In particular, when we choose a {\em finite} subset $B$ as above, the space $\widetilde{\mathcal{S}}_2{}^B$ is a topological subspace of a Euclidean space (of finite dimension $\dim(\mathcal{S}_2)|B|$), and its Euclidean metric transferred to the set $\mathcal{A}^c(\widetilde{\mathcal{S}}_1,\widetilde{\mathcal{S}}_2)$ via the above map $\iota_B$ and the metric $d_{\infty}$ determine the same topology on $\mathcal{A}^c(\widetilde{\mathcal{S}}_1,\widetilde{\mathcal{S}}_2)$, hence on $\widetilde{\mathcal{A}}(\widetilde{\mathcal{S}}_1,\widetilde{\mathcal{S}}_2)$.
\begin{remark}
\label{rem:injective_set_not_closed}
Here we give a remark on Lemma \ref{lem:surjective_set_closed}.
This lemma will fail when we replace \lq\lq surjective'' with \lq\lq injective'' in the statement.
Indeed, in the case $\mathcal{S}_1 = \mathcal{S}_2 = [0,1]$, we define $f_i(x) = x/i$ ($x \in [0,1]$) for integers $i \geq 1$.
Then each $f_i$ is affine and injective, while their limit $f = 0$ with respect to the metric $d_{\infty}$ is not injective.
\end{remark}

\appendix
\section*{Appendix}

\section{Fixed point of affine automorphisms}
\label{sec:appendix_maximally_mixed_state_symmetric_GPT}

In this appendix, we give a proof of the fact (mentioned in Section \ref{subsec:characterization_preliminary}) that there exists a unique fixed point of the $\mathrm{Aut}(\mathcal{S})$-action on $\mathcal{S}$ under the assumption that the convex set $\mathcal{S}$ is compact and finite-dimensional, and $\mathrm{Aut}(\mathcal{S})$ acts transitively on $\mathcal{S}_{\mathrm{ext}}$.
First, we note that the uniqueness of a fixed point is a consequence of Lemma \ref{lem:maximally_mixed_is_interior}.
Indeed, if two distinct fixed points $s_1$ and $s_2$ exist, then the $\mathrm{Aut}(\mathcal{S})$-action fixes the line through $s_1$ and $s_2$, which contains a boundary point of $\mathcal{S}$.
This contradicts Lemma \ref{lem:maximally_mixed_is_interior}.

From now, we prove the existence of a fixed point.
Take a left-invariant integral $\varphi \mapsto \int\!\varphi(g)d g$ on the compact group $G = \mathrm{Aut}(\mathcal{S})$, which is a linear functional on the set of real-valued continuous functions $\varphi:G \to \mathbb{R}$ satisfying the following conditions:
\begin{enumerate}
\item \label{item:invariant_integral_1}
If $\varphi$ is non-negative, then $\int\!\varphi(g)d g \geq 0$.
\item \label{item:invariant_integral_2}
If $\varphi = 1$ constantly, then $\int\!1\,d g = \int\!\varphi(g)d g = 1$.
\item \label{item:invariant_integral_3}
If $h \in G$, then $\int\!\varphi(h g)d g = \int\!\varphi(g)d g$.
\end{enumerate}
Let $\varphi_i:V(\mathcal{S}) \to \mathbb{R}$ denote the $i$-th coordinate function on the Euclidean space $V(\mathcal{S}) = \mathrm{Aff}(\mathcal{S})$.
Fix a point $s_0 \in \mathcal{S}$.
Then the map $\widetilde{\varphi}_i:G \to \mathbb{R}$ defined by $\widetilde{\varphi}_i(g) = \varphi_i(g(s_0))$ ($g \in G$) is continuous by continuity of the $G$-action on $\mathcal{S}$.
Put $v_i = \int\!\widetilde{\varphi}_i(g)d g \in \mathbb{R}$, and let $v \in V(\mathcal{S})$ be the unique element such that $\varphi_i(v) = v_i$ for every $i$.
We prove that $v \in \mathcal{S}$ and it is a fixed point of $G$-action.

As $V(\mathcal{S}) = \mathrm{Aff}(\mathcal{S})$, each $f \in \mathrm{Aut}(\mathcal{S})$ extends to a bijective affine transformation on $V(\mathcal{S})$, which is also denoted by $f$ for simplicity.
First we show that $f(v) = v$ for every $f \in \mathrm{Aut}(\mathcal{S})$.
Note that each $f$ is represented by a matrix $M = (M_{i,j})_{i,j}$ and a vector $b = (b_i)_i$ in such a way that for each $x \in V(\mathcal{S})$, we have $\varphi_i(f(x)) = \sum_{j} M_{i,j} \varphi_j(x) + b_i$ for every $i$.
Now we have
\begin{equation}
\varphi_i(f(v))
= \sum_{j} M_{i,j} v_j + b_i
= \sum_{j} M_{i,j} \int\!\widetilde{\varphi}_j(g)d g + b_i
= \sum_{j} M_{i,j} \int\!\widetilde{\varphi}_j(g)d g + b_i \int\!1\,d g
\end{equation}
where we used the property \ref{item:invariant_integral_2} above in the last equality.
By the linearity, it follows that
\begin{equation}
\begin{split}
\varphi_i(f(v))
= \int\! \left( \sum_{j} M_{i,j} \widetilde{\varphi}_j(g) + b_i \right)\!d g
&= \int\! \left( \sum_{j} M_{i,j} \varphi_j(g(s_0)) + b_i \right)\!d g \\
&= \int\! \varphi_i(f(g(s_0)))d g \\
&= \int\! \widetilde{\varphi}_i(f \cdot g)d g
= \int\! \widetilde{\varphi}_i(g)d g = v_i = \varphi_i(v)
\end{split}
\end{equation}
where we used the property \ref{item:invariant_integral_3} above in the last row.
Hence we have $\varphi_i(f(v)) = \varphi_i(v)$ for every $i$, therefore $f(v) = v$ as desired.

Hence it suffices to show that $v \in \mathcal{S}$.
Assume for contrary that $v \not\in \mathcal{S}$.
Then, as $\mathcal{S}$ is compact and convex, there exists an affine functional $f$ on $V(\mathcal{S})$ that is non-negative on $\mathcal{S}$ and negative at $v$.
Choose a vector $(m_i)_i$ and a value $b$ such that $f(x) = \sum_{j} m_j \varphi_j(x) + b$ for every $x \in V$.
Then a similar argument as above implies that
\begin{equation}
f(v)
= \sum_{j} m_j v_j + b
= \int\! \left( \sum_{j} m_j \varphi_j(g(s_0)) + b \right)d g
= \int\! f(g(s_0))d g \enspace.
\end{equation}
Now we have $f(g(s_0)) \geq 0$ for every $g \in G$ by the choice of $f$ (note that $g(s_0) \in \mathcal{S}$), therefore the right-hand side is non-negative by the property \ref{item:invariant_integral_1} above.
On the other hand, we have $f(v) < 0$ by the choice of $f$ again.
This is a contradiction, therefore we have $f(v) \in \mathcal{S}$ as desired.
Hence the claim in this appendix is concluded.

\section{Proof of Proposition \ref{prop:isometry}}
\label{sec:appendix_proof_isometry}

In this appendix, we give a proof of Proposition \ref{prop:isometry} for the sake of completeness.

\begin{proof}
[Proof of Proposition \ref{prop:isometry}]
Put $G = \mathrm{Aut}(\mathcal{S})$ which is compact.
Then it is well-known from representation theory of compact groups that, given an inner product $\langle \cdot,\cdot \rangle$ on $V = V(\mathcal{S})$, the map $\langle \cdot,\cdot \rangle_G:V \times V \to \mathbb{R}$ defined by $\langle v_1,v_2 \rangle_G = \int_{G} \langle g \cdot v_1, g \cdot v_2 \rangle\,dg$ (where the right-hand side is the Haar integral) is a $G$-invariant inner product on $V$.
Let $e_1,e_2,\dots,e_n \in V$ and $f_1,f_2,\dots,f_n \in V$ be orthonormal bases of $V$ with respect to the inner products $\langle \cdot,\cdot \rangle$ and $\langle \cdot,\cdot \rangle_G$, respectively.
Now we define a bijective linear transformation $\varphi$ on $V$ by $\varphi(f_i) = e_i$ for each $1 \leq i \leq n$.
Put $\mathcal{S}' = \varphi(\mathcal{S})$.
We show that this $\varphi$ satisfies the conditions specified in the statement.
Only the non-trivial part of the claim is that each member of $\mathrm{Aut}(\mathcal{S}')$ is an orthogonal transformation with respect to $\langle \cdot,\cdot \rangle$.
Note that $\mathrm{Aut}(\mathcal{S}') = \{\varphi g \varphi^{-1} \mid g \in G\}$.
Moreover, we have $\langle \varphi(v),\varphi(w) \rangle = \langle v,w \rangle_G$ by the choice of $\varphi$, as $\langle \varphi(f_i),\varphi(f_j) \rangle = \langle e_i,e_j \rangle = \delta_{i,j} = \langle f_i,f_j \rangle_G$ for the basis elements.
Now for each $g \in G$ and $v,w \in V$, we have
\begin{equation}
\langle \varphi g \varphi^{-1}(v), \varphi g \varphi^{-1}(w) \rangle
= \langle g \varphi^{-1}(v), g \varphi^{-1}(w) \rangle_G
= \langle \varphi^{-1}(v), \varphi^{-1}(w) \rangle_G
= \langle v,w \rangle
\end{equation}
where we used the $G$-invariance of $\langle \cdot,\cdot \rangle_G$ in the second equality.
Hence the claim holds.
\end{proof}

\section{On the generality of separated spaces}
\label{sec:appendix_justification_separated_condition}

In this appendix, for the sake of completeness, we briefly describe an argument to show that we may assume without loss of generality that the convex set $\mathcal{S}$ is separated, as mentioned at the beginning of Section \ref{sec:preliminary}.

For an arbitrary convex set $\mathcal{S}$, we define an equivalence relation on $\mathcal{S}$ such that two elements $s_1,s_2 \in \mathcal{S}$ are equivalent if and only if $f(s_1) = f(s_2)$ for every affine functional $f:\mathcal{S} \to \mathbb{R}$ with bounded image $f(\mathcal{S})$.
Then the corresponding quotient set becomes a separated convex set, with essentially the same set of the bounded affine functionals as the case of $\mathcal{S}$.
Therefore, by considering the quotient set instead of $\mathcal{S}$ itself, it suffices to study convex sets which are separated.
Hence the claim of this appendix follows.

\paragraph*{Acknowledgements}
A part of this paper was presented at 14th Workshop on Quantum Information Processing (QIP 2011), Singapore, January 10--14, 2011.
This work was partially supported by Grant-in-Aid for Young Scientists (B) (No.20700017 and No.22740079), The Ministry of Education, Culture, Sports, Science and Technology (MEXT), Japan.
Moreover, the authors thank the anonymous referees in several submissions of the paper for their precious comments.

\end{document}